\documentclass[letter,10pt]{article}
\usepackage[utf8]{inputenc}
\usepackage{amssymb}
\usepackage{amsthm}
\usepackage{amsmath}
\usepackage{mathrsfs}
\usepackage{verbatim}
\usepackage{enumerate}

\usepackage{cite}
\usepackage{hyperref}
\usepackage{cleveref}
\usepackage{changes}
\usepackage{color}
\newtheorem{theorem}{Theorem}
\newtheorem{corollary}[theorem]{Corollary}
\newtheorem{lemma}[theorem]{Lemma}
\newtheorem{proposition}[theorem]{Proposition}
\newtheorem{remark}[theorem]{Remark}
\newtheorem{definition}[theorem]{Definition}
\newtheorem{example}[theorem]{Example}

\newcommand{\R}{\mathbb{R}}

\newcommand{\U}{\mathcal{U}}

\DeclareMathOperator{\id}{id}
\newcommand{\X}{X}
\newcommand{\Y}{Y}

\newcommand\Lone{{\mathrm{L}}^1}

\def\<#1,#2>{\left<#1,#2\right>}
\let\bar\overline

\def\PP{{\cal P}}

\title{{Variational problems involving unequal dimensional optimal transport}\footnote{The authors want to thank the anonymous referees for their observations and careful reading of the paper. B.P. is pleased to acknowledge the support of National Sciences and Engineering Research Council of Canada Discovery Grant numbers 412779-2012 and 04658-2018. L.N. would like to acknowledge the support from the project MAGA ANR-16-CE40-0014 (2016-2020).}}

\author{
Luca Nenna\thanks{Laboratoire de Mathématiques d’Orsay, Univ. Paris-Sud, CNRS, Université Paris-Saclay, 91405 Orsay, France.  \texttt{luca.nenna@math.u-psud.fr}}  
\and 
Brendan Pass\thanks{Department of Mathematical and Statistical Sciences, 632 CAB, University of Alberta, Edmonton, Alberta, Canada, T6G 2G1  \texttt{pass@ualberta.ca}} 
}

\begin{document}

\maketitle
 
 \begin{abstract}
 	This paper is devoted to variational problems on the set of probability measures which involve optimal transport between unequal dimensional spaces.  In particular, we study the minimization of a functional consisting of the sum of a term reflecting the cost of  (unequal dimensional) optimal transport between one fixed and one free marginal, and another functional of the free marginal (of various forms).  Motivating applications include Cournot-Nash equilibria where the strategy space is lower dimensional than the space of agent types.  For a variety of different forms of the term described above, we show that a nestedness condition, which is known to yield much improved tractability of the optimal transport problem, holds for any minimizer.  Depending on the exact form of the  functional, we exploit this to find local differential equations characterizing solutions, prove convergence of an iterative scheme to compute the solution, and prove regularity results.
 	
 \end{abstract}
 \tableofcontents

 \section{Introduction}
 This paper is devoted to the study of functionals of the form
 \begin{equation}\label{eqn: variational problem}
  \mathcal{J}(\mu,\nu) = \mathcal {T}_c(\mu,\nu)+\mathcal{F}(\nu) + \mathcal{G}(\mu)
  \end{equation}
  depending on probability measures $\mu \in P(\bar \X)$, $\nu \in P(\bar \Y)$ on the closures of open, bounded domains $X \subseteq \mathbb{R}^m$, $Y \subseteq \mathbb{R}^n$, where  $\mathcal {T}_c(\mu,\nu)$ is the optimal transport distance between $\mu$ and $\nu$ induced by the cost function $c: \bar \X \times \bar \Y \rightarrow \mathbb{R}$:
 
 \begin{equation}\label{eqn: ot problem}
 \mathcal {T}_c(\mu,\nu) = \inf_{\gamma \in \Pi(\mu,\nu)}\int_{\X \times \Y}c(x,y)d\gamma(x,y).
 \end{equation}
 Here, $\Pi(\mu,\nu)$ is the set of all probability measures on the product space $X \times \Y$ whose marginals are $\mu$ and $\nu$.  
 
 We are interested in characterizing the minimizers of $(\mu,\nu) \mapsto \mathcal{J}(\mu,\nu)$, as well as the minimizers of the subproblems obtained when either $\mu$ or $\nu$ is fixed:  $\nu \mapsto \mathcal{J}(\mu,\nu)$ and $\mu \mapsto \mathcal{J}(\mu,\nu)$.  Problems of these general forms, for various choices of the functionals $\mathcal{F}$ and $\mathcal{G}$, arise in a wide variety of applications, including: gradient flows on Wasserstein space (where the minimization $\nu \mapsto \mathcal{J}(\mu,\nu)$ represents one step in a discrete gradient flow), displacement interpolation (when $\mathcal{F}$ is the Wasserstein distance to a second probability measure) Cournot-Nash equilibria in game theory, city planning problems, hedonic pricing in economics \cite{Carlier_wasserstein_barycenter,abgcmor,abgcptrl,buttazzo2006three,buttazzo2009mass,buttazzo2005model,carlier2003optimal,carlier2010matching,carlier2005variational}, and have consequently received a fair bit of attention in the literature.  Most of the analytical progress so far, however, has been restricted to the case where the dimensions of the spaces $X$ and $\Y$ coincide, $m=n$.  In a wide variety of applications, particularly in economics and game theory, however, these dimensions may differ: in the Cournot-Nash problem, for example, $X$ parameterizes a space of agents, (differentiated by $m$ characteristics $x_1,x_2,...,x_m$ of agent $x =(x_1,...,x_m)$) while $\Y$ represents a space of strategies.  The dimensions of these spaces reflect the number of characteristics used to differentiate among agents, and the number of parameters involved in the choice of strategy, respectively, and need not be the same in general.

 Minimizers of \eqref{eqn: variational problem} when $m=n$ have been studied extensively; for the subproblems where one marginal is fixed, under mild conditions on the functionals $\mathcal{F}$ and $\mathcal{G}$, existence and uniqueness of minimizers has been established, and, depending on the precise forms of $\mathcal{F}$ and $\mathcal{G}$, various regularity results and bounds on solutions exist.   Solutions can be characterized by partial differential equations, and various numerical schemes for solving them have been proposed (see \cite[chapter 6 and section 7.4]{santambook} and the references therein and \cite{Blanchet2017,abgcptrl,abgcmor,Carlier-NumericsBarycenters,peyre2017computational,peyregradientflows}).

 Our present purpose is to initiate the analysis of minimizers of $\mathcal{J}$ (often with $\mu$ or $\nu$ fixed) when $m>n$.  Motivated by matching problems in economics \cite{ChiapporiMcCannPass15p}, the second named author together with collaborators has recently introduced conditions under which the unequal dimensional optimal transport problem \eqref{eqn: ot problem} is relatively tractable \cite{PassM2one}\cite{mccann2018optimal}.  More precisely, when the target $\Y$ is unidimensional ($m>n=1$), the condition, known as \emph{nestedness}, allows one to construct almost closed form solutions from the cost function $c$ and marginals $\mu$ and $\nu$ \cite{PassM2one}.  For higher dimensional targets, ($m>n>1$), an analogous condition ensures that a certain, generally non-local, partial differential equation on the lower dimensional space $\Y$ characterizing the solution is, in fact, local and degenerate elliptic \cite{mccann2018optimal}. These conditions, when present, are powerful tools for analyzing solutions; in particular, nestedness is conjectured in \cite{PassM2one} to be necessary for the continuity of optimal maps (the high dimensional version's necessity and sufficiency for potentials to be $C^2$ and strongly elliptic was verified in \cite{mccann2018optimal}, under mild topological conditions) and computation of solutions to nested problems is presumably much simpler than non-nested ones.
 
 The nestedness condition, and its higher dimensional counterpart (which we will  refer to as generalized nestedness hereafter) are \textit{joint} conditions on the cost $c$ and marginals $\mu$ and $\nu$, whereas in the present context, only the cost and \textbf{one} of the marginals (\textbf{neither} in the case of double minimizations) is prescribed; the other marginal is part of the solution to the problem.
 
  We prove here that, under various conditions on $c$, $\mu$ and $\Y$,  $(c,\mu,\nu)$ is nested whenever $\nu$ minimizes $\nu \mapsto \mathcal{J}(\mu,\nu)$, for a variety of different choices of the functional $\mathcal{F}$; analogous results for certain specific forms of $\mathcal G$ are also established for minimizations on the higher dimensional space, $\mu \mapsto \mathcal{J}(\mu,\nu)$, and for double minimizations $(\mu,\nu) \mapsto \mathcal{J}(\mu,\nu)$.  We go on to demonstrate that this a priori guarantee of nestedness makes the problem of characterizing or identifying  the minimizers much more tractable; in different contexts, depending on the precise form of $\mathcal F$, we establish that solutions can be characterized by (local) differential equations, can be computed numerically by a convergent iterative scheme, or can be derived in almost closed form.  In a forthcoming companion paper \cite{NennaPass2} focusing on the Cournot-Nash problem pioneered by Blanchet-Carlier \cite{abgcmor, blanchet2014remarks, abgcptrl}, we exhibit completely solved examples, in which the conditions ensuring nestedness are verified and the resulting differential equation solved numerically.

 \paragraph{Main results.}
 The main results we establish in the paper are the following:
 \begin{itemize}
 \item Theorem  \ref{lemma: nestedness condition}  characterizes nestedness for a given model $(c,\mu,\nu)$ when the target is one dimensional ($n=1$).  Its consequences include Corollaries \ref{cor: sufficient conditions for nestedness} and \ref{cor: nestedness from bounds on mu}, which give sufficient conditions for nestedness in terms of either a lower bound on $\nu$, which depends on $c$ and $\mu$,  or an upper bound on $\mu$, depending on $c$ and $\nu$.
 \item Theorem \ref{nestedness_congestion} ensures nestedness of the model when the lower dimensional measure $\nu$ minimizes $\nu \mapsto \mathcal J (\mu,\nu)$, $\mathcal F$ is a congestion term and the target is one dimensional;
 
 \item Theorem \ref{lemma: nestedness from interaction} provides  sufficient conditions for generalized nestedness when the lower dimensional measure $\nu$ minimizes $\nu \mapsto \mathcal J (\mu,\nu)$ (with the target dimension $n$ not necessarily one) and the functional $\mathcal F$ is composed of interaction and potential terms;
 
 \item Theorem \ref{bestReply} generalizes a previous result in \cite{blanchet2014remarks}, providing convergence of the best reply scheme to compute minimizers of $\nu \mapsto \mathcal J (\mu,\nu)$  when $\mathcal F$ is made of interaction and potential terms;
 
 \item Theorem \ref{high_dimensional_congestion} proves nestedness when the higher dimensional measure $\mu$ minimizes $\mu \mapsto \mathcal J (\mu,\nu)$, $\mathcal G$ is a congestion functional and the target is one dimensional $(n=1)$;
 
 \item Theorem \ref{double_minimization}  provides conditions guaranteeing nestedness for minimizers $(\mu,\nu)$ of the double minimization problem $(\mu, \nu) \mapsto \mathcal J (\mu,\nu)$, when $\mathcal F$ is composed of interaction and potential terms, $\mathcal G$ is of congestion type, and the target is one dimensional;
 
 \item  Theorem \ref{hedonical_nestedness} considers the case in which the functional $\mathcal F(\nu) :=T_{c_2}(\mu_2,\nu)$ reflects the cost of optimal transport to a second fixed measure, as is the case in the hedonic pricing problem in economics, so that, after relabeling $\mu$ and $c$ from \eqref{eqn: variational problem} as $\mu_1$ and $c_1$, the goal is to minimize $\nu \mapsto T_{c_1}(\mu_1,\nu)+T_{c_2}(\mu_2,\nu)$. When the target is one dimensional, the theorem implies that under a  condition we call \textit{hedonic nestedness} 
  one can find  solutions $\nu$ to this problem almost explicitly.  Among other consequences, this condition ensures nestedness of  $(c_i, \mu_i,\nu)$ for both $i=1$ and $2$.
 \end{itemize}

 \paragraph*{Organization of the paper.}
In section 2, we recall the basics of unequal dimensional optimal transport and, for one dimensional targets $n=1$, establish a sufficient condition for nestedness  which relies only on bounds on the densities $\bar \mu$ and $\bar \nu$ of $\mu(x)=\bar \mu(x)dx$ and $\nu(x)=\bar \nu(y)dy$, rather than complete knowledge of these marginals.  This result will be used in subsequent sections, but it may also be of independent interest.  In section 3, we establish nestedness of solutions for several different forms of the functional $\mathcal F$: congestion terms on either the higher or lower dimensional domain (when the target is one dimensional), as well as potential and interaction terms on the lower dimensional domain. Section 4 is then reserved for the analysis of hedonic pricing problems, where, like the first term in \eqref{eqn: variational problem},  $\mathcal{F}(\nu) =\mathcal{T}_{c_2}(\mu_2, \nu)$ reflects optimal transport to a second higher dimensional marginal.

 \section{Optimal Transport between unequal dimensions}
 \label{recall}
We recall here some basic facts about the optimal transportation problem \eqref{eqn: ot problem}.
 
 
Assuming that $c:\bar \X\times\bar \Y\rightarrow\R$ is bounded and continuous then problem \eqref{eqn: ot problem} always admits at least one solution.  We will assume throughout this paper that $c \in C^2(\bar \X \times \bar \Y)$ satisfies the \emph{twist} condition, which asserts that for each $x \in \X$
\begin{equation} 
y \mapsto D_xc(x,y) \text{ is injective on } \Y,
\end{equation}
as well as the \emph{non-degeneracy} condition, asserting that the $m \times n$ matrix $D^2_{xy}c(x,y)$ of mixed second order partial derivatives has full rank for each $(x, y) \in \bar \X \times \bar \Y$ (as $m >n$, this means $D^2_{xy}c(x,y)$ has rank $n$).



Before discussing the case $m>n$, it is important to highlight that the Monge-Kantorovich problem  admits a dual formulation which is useful in order to understand the solution to \eqref{eqn: ot problem} 

\begin{equation}
 \mathcal {T}^{dual}_c(\mu,\nu):=\sup_{(u,v)\in\U}  \displaystyle \int_{\X} u(x)d\mu(x)+\int_{\Y} v(y)d\nu(y),
\end{equation}
where $\U:=\{ (u,v)\in\Lone(\mu)\times\Lone(\nu)\;:\; u(x)+v(y)\leq c(x,y)\;\text{on}\;\X\times\Y\}$. 
The interesting fact is that $ \mathcal {T}_c(\mu,\nu)= \mathcal {T}^{dual}_c(\mu,\nu)$.


Under mild conditions (for instance, if $c$ is differentiable, as we are assuming here, $X$ connected and $\bar \mu(x) >0$ for all $x \in X$ \cite{santambook}), there is a unique solution $(u,v)$ to the dual problem, up to the addition $(u,v)\mapsto(u+C,v-C) $ of constants adding to $0$, known as the Kantorovich potentials, and these potentials are $c$-concave; that is, they satisfy
$$
u(x) =v^c(x):=\min_{y \in \Y}[c(x,y) -v(y)], \text{ }v(y) =u^c(y):=\min_{x \in \X}[c(x,y) -u(x)],
$$

  This solution is used to define generalized nestedness.  The following definition is slightly adapted from \cite{mccann2018optimal}.\footnote{The almost everywhere differentiability of $v(y)$ holds \textit{automatically} if $\nu$ is absolutely continuous with respect to Lebesgue measure, as is assumed in \cite{mccann2018optimal}; since we will at times want to establish results without this assumption here, we add the differentiability to the definition of nestedness.  Note that in \cite{mccann2018optimal}, a strengthening of the condition below, involving a second order condition on $v$ is also considered.  Since this plays no role in the present work, we omit it.}
\begin{definition}\label{def: high dim nestedness}
When $m>n$, we will say that the model $(c, \mu, \nu)$  satisfies the generalized \textit{nestedness} condition if for $\nu$ almost every $y$ the potential $v$ is differentiable and we have
\begin{equation}\label{eqn: high dimensional nestedness}
\begin{split}
\partial^cv(y)&:=\{x: u(x)+v(y) =c(x,y)\}\\&  =X_=(y,Dv(y)) := \{x: Dv(y) =D_yc(x,y)\}.
\end{split}
\end{equation}
\end{definition}
The containment $\partial^cv(y)\subseteq X_=(y,Dv(y))$ holds automatically throughout the domain of $Dv(y)$; it is therefore the opposite containment that distinguishes nested from non-nested models.   The origin of the term nestedness lies in the one dimensional target setting, in which case the  condition is equivalent to nestedness of certain \textit{super-level} sets of $x \mapsto \frac{\partial c}{\partial y}(x,y)$ \cite{PassM2one}.  This is discussed in more detail below (see Proposition \ref{nested}).

If both $\nu$ and $\mu$ are absolutely continuous, the potential $v$ satisfies the Monge-Ampere type equation almost everywhere \cite{mccann2018optimal}\footnote{Note that, here and below, our notation differs somewhat from \cite{PassM2one} and \cite{mccann2018optimal}, since we have adopted the convention of minimizing, rather than maximizing, in \eqref{eqn: ot problem}.}:

 \begin{equation}
\label{eqn: unequal MA}
\bar\nu(y)=\int_{\partial^cv(y)}\dfrac{\det(D_{yy}^2c(x,y)-D^2v(y))}{\sqrt{|\det(D^2_{yx}c D^2_{xy}c)(x,y)|}}\bar\mu(x)d\mathcal H^{m-n}(x),
\end{equation}
where, as before, $\bar \mu(x): =\frac{d \mu}{d x}(x)$ and $\bar \nu(y): =\frac{d \nu}{d y}(y)$ are the densities of $\mu$ and $\nu$.  In general,  this is a \emph{non-local} differential equation for $v(y)$, since the domain of integration $\partial^cv(y)$ is defined using the values of $v$ and $u=v^c$ throughout $\Y$; however, when the model satisfies the generalized nestedness condition   (namely $\partial^cv(y) =X_=(y,Dv(y))$), it reduces to the local equation \cite{mccann2018optimal}:

 \begin{equation}
\label{eqn: local unequal MA}
\bar\nu(y)=\int_{X_=(y,Dv(y))}\dfrac{\det(D_{yy}^2c(x,y)-D^2v(y))}{\sqrt{|\det(D^2_{yx}c D^2_{xy}c)(x,y)|}}\bar\mu(x)d\mathcal H^{m-n}(x).
\end{equation}

\subsection{Multi-to one-dimensional optimal transport}
We consider now the optimal transport problem in the case in which $m>n=1$ (for more details we refer the reader to \cite{PassM2one}).  In this case,  generalized nestedness follows from  a relatively simple condition, related to the following heuristic attempt to construct solutions to \eqref{eqn: ot problem}.
Let us define the level and super-level sets of $D_yc$ as follows 
$$
X_=(y,k) :=\{x\in\X\;:\; \frac{\partial c}{\partial y}(x,y)= k \},
$$
\begin{equation*}
X_\geq (y,k):=\{x\in\X\;:\; \frac{\partial c}{\partial y}(x,y)\geq k \},
\end{equation*}
 as well as the strict variant $X_> (y,k):=X_\geq (y,k)\setminus X_=(y,k)$.
 In order to build an optimal transport map $T$, we take the unique level set splitting the mass proportionately with $y$; that is, defining $k(y)$ such that
 \begin{equation}\label{eqn: def of k}
  \mu(X_\geq(y,k(y)))=\nu((-\infty,y]),    
  \end{equation}
 then we set $y=T(x)$ for all $x$ which belong to $X_=(y,k(y))$.
 Notice that if there exists $x\in\X$ such that $x\in X_=(y_0,k(y_0)) \cap X_=(y_1,k(y_1))$ then the map $T$ is not well-defined.  The absence of such a degenerate case is equivalent to the super-level sets being nested; this is the definition of nestedness from \cite{PassM2one}, which implies the more general Definition \ref{def: high dim nestedness} when $n=1$, as the following result from \cite{mccann2018optimal} affirms.

 
\begin{proposition}[Nestedness for one dimensional targets]
\label{nested}
The model $(c,\mu,\nu)$  satisfies the generalized nestedness condition if 
\begin{equation}\label{eqn: 1-d nestedness}
\forall y_0,y_1\; \text{with} \; y_1>y_0,\; \nu([y_0,y_1])>0\Longrightarrow X_\geq (y_0,k(y_0)) \subseteq X_>(y_1,k(y_1)) .
\end{equation}
\end{proposition}
 We will say $(c,\mu,\nu)$ is \textit{nested} if \eqref{eqn: 1-d nestedness} is satisfied.

 If the model $(c,\mu,\nu)$ is nested then \cite{PassM2one}[Theorem 4] assures that $\gamma_T=(\id,T)_\sharp \mu$, where the map $T$ is built as above, is the unique minimizer of \eqref{eqn: ot problem} in $\Pi(\mu,\nu)$. Moreover, the optimal potential $v(y)$ is given by $v(y)=\int_{-\infty}^y k(t) dt$, and so \eqref{eqn: local unequal MA} becomes
 
 \begin{equation}
 \label{MAmulti2one}
 \bar\nu(y)=\int_{X_=(y,k(y))}\dfrac{D^2_{yy}c(x,y)-k'(y)}{|D^2_{xy}c(x,y)|}\bar\mu(x)d\mathcal H^{m-1}(x),
 \end{equation}
 where $|\cdot|$ denotes the standard Euclidean norm.
 \subsection{A sufficient condition for nestedness}\label{sect: bounds imply nestedness}
 The nestedness condition \eqref{eqn: 1-d nestedness} and its higher dimensional generalization in Definition \ref{def: high dim nestedness} depend on all the data, $(c, \mu, \nu)$, of the optimal transport problem \eqref{eqn: ot problem}. 
In the following we give sufficient conditions in the $m >n=1$ setting which let us establish nestedness when, instead of knowing  $\nu$ (respectively $\mu$), we know only bounds on its density.  This will be useful in subsequent sections when we study variational problems of form \eqref{eqn: variational problem}, as for appropriate choices of $\mathcal{F}$ and $\mathcal{G}$, minimizers have upper and/or lower bounds.  Our approach here is somewhat reminiscent of the approach in \cite{Pass2012}, where specific local comparisons (depending precisely on the marginals) between the masses of particular sets were used to, essentially, ensure nestedness.  However, sufficient conditions for nestedness without complete knowledge of the marginals were not formulated in \cite{Pass2012}.

 Fix $y_0 <y_1$ (where $y_0,y_1\in\Y$), $k_0 \in D_yc(\X,y_0)$ and set $k_{max}(y_0,y_1,k_0)=\sup\{k: X_\geq(y_0,k_0) \subseteq X_\geq (y_1,k)\}$.  We then define the \emph{minimal mass difference}, $D^{min}_\mu$, as follows: 
 
 $$
D^{min}_\mu(y_0,y_1,k_0) =\mu(X_\geq(y_1,k_{max}(y_0,y_1,k_0))\setminus X_\geq(y_0,k_0)).
 $$
 The minimal mass difference represents the smallest amount of mass that can lie between $y_0$ and $y_1$, and still have the corresponding level curves $X_=(y_0,k_0)$ and $X_=(y_1,k_1)$ not intersect.
 In the following we assume that $d\mu(x) =\bar \mu (x)dx$ and $d\nu(y) =\bar \nu (y)dy$ are absolutely continuous with respect to the Lebesgue measure.
 \begin{theorem}\label{lemma: nestedness condition}
 	Assume that $\mu$ and $\nu$ are absolutely continuous with respect to Lebesgue measure.     If $D^{min}_\mu(y_0,y_1,k(y_0)) < \nu([y_0,y_1])$ for all $y_0<y_1$ where $k(y)$ is defined by \eqref{eqn: def of k}, then $(c,\mu,\nu)$ is nested.  Conversely, if $(c,\mu,\nu)$ is nested, we must have $D^{min}_\mu(y_0,y_1,k(y_0)) \leq \nu([y_0,y_1])$ for all $y_1 > y_0$.
 \end{theorem}
 \begin{proof}
 	This essentially follows from the definition of nestedness \eqref{eqn: 1-d nestedness}; if nestedness fails, we have that $X_\geq(y_0,k_0)$ is not contained in $X_>(y_1, k_1)$ for some $y_0 <y_1$, with $k_i =k(y_i)$.  Therefore, $k_1 \geq k_{max}(y_0,y_1,k_0)$, and so
 	$$
 	\mu(X_\geq(y_1,k_1)\setminus X_\geq(y_0,k_0)) \leq D ^{min}_\mu(y_0,y_1,k_0).
 	$$
 	Now, by definition of the $k_i$, $\nu([y_0,y_1]) = \mu(X_\geq(y_1,k_1)) -\mu(X_\geq(y_0,k_0)) \leq 	\mu(X_\geq(y_1,k_1)\setminus X_\geq(y_0,k_0)) \leq D^{min}_\mu(y_0,y_1,k_0)$, which contradicts the assumption in the Lemma.
 	
 	On the other hand, if $(c,\mu,\nu)$ is nested, then for all $y_1>y_0$,
 	$$
 	\mu(X_\geq(y_1,k_1) \setminus X_\geq(y_0,k_0)) =\mu(X_\geq(y_1,k_1)) -\mu( X_\geq(y_0,k_0)) =\nu([y_0,y_1]),
 	$$
with $k_i=k(y_i)$ as before. 	Since nestedness also implies $k_1 \leq k_{max}(y_0,y_1,k_0)$, we have 
 	$$
 	D^{min}_\mu(y_0,y_1,k_0) \leq \mu(X_\geq(y_1,k_1) \setminus X_\geq(y_0,k_0)) =\nu([y_0,y_1]),
 	$$
 	completing the proof.
 \end{proof}	
 
 The following sufficient condition is more convenient when one has bounds on the density $\bar \nu (y)$, as will be the case in the next section.
 \begin{corollary}\label{cor: sufficient conditions for nestedness}
 	If for each $y_0 \in \Y$, we have
 	$$
 	\sup_{y_1 \in \Y,y_0\leq y\leq y_1}\left[\frac{D^{min}_\mu(y_0,y_1,k(y_0))}{y_1-y_0}-\bar\nu(y)\right] < 0,
 	$$
 	where $\bar\nu =\frac{d\nu(y)}{dy}$, then $(c,\mu,\nu)$ is nested.	
 	
 \end{corollary}

 \begin{proof}
 	The condition means that for each $y \in [y_0,y_1]$ we have
 	$$
 	\frac{D^{min}_\mu(y_0,y_1,k(y_0))}{y_1-y_0} < \bar\nu(y).
 	$$
 	Integrating $y$ from $y_0$ to $y_1$ and applying Theorem \ref{lemma: nestedness condition} yields the desired result.
 \end{proof}	
 As a consequence, if the quantity $\frac{D^{min}_\mu(y_0,y_1,k_0)}{y_1-y_0}$ is bounded above for all $y_0<y_1$ and $k_0 \in D_yc(X,y_0)$, then a corresponding lower bound on $\bar\nu$ will ensure nestedness.  We illustrate this with an example.  
 \begin{example}	\label{ex: disk to arc}
 	Letting $\mu$ be uniform measure on the quarter disk, $\X:=\{x_1,x_2>0: x_1^2+x_2^2 < 1 \} \subseteq \mathbb{R}^2$, so that $d\mu(x) =\frac{4}{\pi}dx$, $\Y=(0,\bar y)$, with $\bar y \leq \frac{\pi}{2}$, parametrize a segment of the unit circle, and $c(x,y) =-(x_1\cos(y) +x_2\sin(y))$ the bilinear cost, we note that the level curves $X_=(y,k)$ are line segments parallel to the line segment joining $(0,0)$ with $(\cos(y),\sin(y))$.  
 	
 	Now, for fixed $y_0<y_1,k_0$,   the lines $X_=(y_0,k_0)$ and $X_=(y_1,k_{max}(y_0,y_1,k_0))$ intersect on the $x_2$ axis, if $k_0<0$ and the $x_1$ axis if $k_0>0$. In either case, the region
 	$X_\geq (y_1,k_{max}(y_0,y_1,k_0)) \setminus X_\geq(y_0,k_0)$ is the part of the wedge of angle $y_1-y_0$ between the two lines which intersects $\X$;  this wedge is smaller than the corresponding wedge $X_\geq (y_1,0) \setminus X_\geq(y_0,0)$, for which the intersection point is at the origin.  Therefore,
 	
 	$$
 	D^{min}_\mu(y_0,y_1,k_0)  \leq \mu (X_\geq (y_1,0) \setminus X_\geq(y_0,0)) =(y_1-y_0) \frac{2}{\pi}
 	$$
 	It therefore follows from Corollary \ref{cor: sufficient conditions for nestedness} that the model $(c,\mu,\nu)$ is nested for any $d\nu=\bar\nu(y)dy$ such that 
 	\begin{equation}\label{eqn: sufficient lower bound}
 	\bar\nu(y) >  \frac{2}{\pi}.
 	\end{equation} 
 \end{example}
\begin{remark}
Suppose that $c$ is of pseudo-index form; that is, $c(x,y) = C(I(x), y) +B(x)$, where $I: \mathbb{R}^m \rightarrow \mathbb{R}$, $C:\mathbb{R} \times \mathbb{R} \rightarrow \mathbb{R}$ and $B: \mathbb{R}^m \rightarrow \mathbb{R}$ are smooth, $DI \neq 0$ and $\frac{\partial ^2 C}{\partial I\partial y} <0$.    In this case, $k_{max}(y_0,y_1,k_0)=D_yC(I,y_1)$, where $I$ is such that $D_yC(I,y_0) =k_0$, and so $X_\geq(y_0,k_0) =X_\geq(y_1,k_{max}(y_0,y_1,k_0))$ and $D^{min}_\mu(y_0,y_1,k_0) =0$ for any a.c $\mu$.  
	Therefore, by Theorem \ref{lemma: nestedness condition},  we recover the following fact from \cite{PassM2one}:  $(c,\mu,\nu)$ is nested for any choices of marginals $\mu$ and $\nu$, provided $\nu$ charges every interval.
\end{remark}

 When considering problems where the measure $\nu$ on $\mathbb{R}$ is fixed but the high dimensional marginal $\mu$ is not, the following reformulation is sometimes convenient; it implies that an appropriate upper bound on $\bar \mu$ yields nestedness.
 \begin{corollary}\label{cor: nestedness from bounds on mu}
 	Suppose that for all $y_0 \in \Y$ and ,
 	$$
	\sup_{y_1 \in \Y,y_0\leq y\leq y_1} [||\bar\mu||_{L^{\infty}(X_\geq(y_1,k_{max}(y_0,y_1,k(y_0)))\setminus X_\geq(y_0,k(y_0)))} \frac{D^{\min}_{vol}(y_0,y_1,k(y_0))}{y_1-y_0}-\bar\nu(y)] < 0.
 	$$
 	Then $(c,\mu,\nu)$ is nested.
 \end{corollary}
\begin{proof}
First of all notice that the following holds
$$D^{min}_\mu(y_0,y_1,k(y_0))  \leq ||\bar\mu||_{L^{\infty}(X_\geq(y_1,k_{max}(y_0,y_1,k(y_0)))\setminus X_\geq(y_0,k(y_0)))}D^{min}_{vol}(y_0,y_1,k(y_0)).$$
Then by Corollary \ref{cor: sufficient conditions for nestedness} we can conclude.
\end{proof}

 \section{Some variational problems}
We now turn our focus to minimizing functionals of the form \eqref{eqn: variational problem}, and the subproblems obtained when one of the measures is fixed: 
\begin{equation}
\label{minPbNu}
\min \{\mathcal T_c(\mu, \nu) +\mathcal{F}(\nu)\;:\;\nu\in\PP(\bar \Y) \},
\end{equation}
and
    \begin{equation}
\label{minPbMu}
\min \{\mathcal T_c(\mu, \nu) +\mathcal{G}(\mu)\;:\;\mu\in\PP(\bar \X) \}.
\end{equation}
    
    The approach we take depends strongly on the form of the functionals; when $n=1$, and $\mathcal{F}$ (respectively $\mathcal{G}$) has a congestion form, we can derive upper and lower bounds on solutions $\nu$ (respectively $\mu$) to \eqref{minPbNu} (respectively \eqref{minPbMu}), which yield nestedness by the results in the last section.  These bounds are established in subsection \ref{subsect: density bounds} below, and applied in subsections \ref{lbwi} and \ref{subsect: min high dim}.   Notice that these bounds on the unknown densities actually do not depend on the dimension of $\X$ and $\Y$. The dimension plays a  role only when we combine these bounds with the results of the previous section, which require one dimensional targets. For functionals with different forms, such bounds are not available; in these cases, one can sometimes derive nestedness in \eqref{minPbNu} directly from the optimality conditions, even for higher dimensional targets.  We follow this approach in subsection \ref{subsect: interaction and potential}. 
    \subsection{Bounds on densities of solutions}\label{subsect: density bounds}
 
  Let us consider \eqref{minPbNu} 
  where $\mathcal{F}$ is of \emph{congestion}, or internal energy, form: $\mathcal{F}(\nu):=\int_{\Y} f(\bar \nu(y))dy$
    with $f:[0,\infty) \rightarrow \mathbb{R}$  continuously differentiable on $(0,\infty)$, strictly convex with superlinear growth at infinity, satisfying,
 $$
 \lim_{\bar \nu \rightarrow 0^+}f'(\bar \nu) =-\infty.
 $$  
A prototypical example is  the entropy, $f(\bar \nu) =\bar \nu\ln(\bar \nu)$. This is a popular type of functional in a variety of settings; in the Cournot-Nash case, it reflects agents' desires to choose strategies that are not too close to each other.\footnote{Blanchet-Carlier motivate congestion functionals via a holiday choice example, where $y \in Y$ parametrizes a location where agent might go for vacation \cite{abgcmor}.  In the present context, interpreting $x \in \X \subseteq \R^2$ as the home location of agent $x$, and taking $c(x,y) =|x-y|^2$, a curve $Y \subseteq \mathbb{R}^2$ might parameterise a one dimensional continuum of desirable or easily accessible holiday locations (the coast of an ocean or lake, or locations along a major railroad or highway, for example).  Agents therefore try to minimize the distance from their holiday location to their homes, as well as the density of other agents vacationing in the same place.}

Notice that the assumptions on $\mathcal F$ guarantees the existence of a minimizer of \eqref{minPbNu}.

In the proposition below, we let  $M_c :=\sup_{(x,y_0,y_1) \in (\bar \X \times \bar \Y \times \bar Y)}\frac{|c(x,y_0) -c(x,y_1)|}{|(x,y_0) -(x,y_1)|}$, where $|\cdot|$ denotes the Euclidean norm, be a global Lipschitz constant for $y\mapsto c(x,y)$ and for a bounded real valued function $v:\Y \rightarrow \mathbb{R}$, denote by $K_v:(0,\infty)\rightarrow (-\infty, \infty)$ the inverse of the  monotone increasing function $z \mapsto \int_Y (f')^{-1}(z-v(y))dy: (-\infty, \infty) \rightarrow (0, \infty)$.

For simplicity, we assume below that $ 0 \in \bar \Y$.  The argument below is inspired by \cite[Section 7.4.1]{santambook}.

\begin{proposition}\label{prop: upper bound on density from conjestion}
	The minimizing $\nu$ in \eqref{minPbNu} is absolutely continuous with respect to Lebesgue, and its density $\bar \nu$ satisfies 
	$$
	(f')^{-1}(K_{-M_c|y|}(1)-M_c|y|) \leq \bar \nu(y) \leq (f')^{-1}(K_{M_c|y|}(1) + M_c|y|)
	$$
\end{proposition}
\begin{proof}
Let $\nu$ be a solution to \eqref{minPbNu}. The absolute continuity follows immediately from the form of $\mathcal F$ (by convention, $\mathcal F[\nu] = \infty$ for non a.c. $\nu$).  It was shown in \cite{abgcmor} that $\bar \nu >0$ throughout $\Y$, and the following equation  holds $\nu-$a.e
 
 $$
 v(y) +f'(\bar \nu)=C,
 $$
for some constant $C$, where $v(y)$ is the Kantorovich potential for transport between $\nu$ and $\mu$.   Without loss of generality we take $v(0) =0$.  We then have
 
 $$
 \bar \nu(y) =(f')^{-1}(C-v(y)).
 $$
 As $\nu$ is a probability measure, we have
 $$
 1= \int_{\Y}\bar \nu(y)dy = \int_Y (f')^{-1}(C-v(y))dy
 $$
 As $f$ is strictly convex, the right hand side of this equation is strictly monotone in $C$ and therefore the equation determines $C =K_v(1)$ uniquely.
 
 Now, it is well known that, as a Kantorovich potential for the cost $c$, $v$ is 
 Lipschitz, with constant $M_c$ \cite{McCann2001}, and so $-M_c  |y| \leq v(y) \leq M_c |y|$ (recall that $v(0)=0$).  This produces a lower bound on $C$ via:
 $$
 1 \leq  \int_Y (f')^{-1}(C+M_c|y|)dy
 $$
 or
 $$
 C\geq K_{-M_c|y|}(1).
 $$
 It then follows that 
 $$
 \bar \nu(y) =(f')^{-1}(C-v(y)) \geq  (f')^{-1}(K_{-M_c|y|}(1)-M_c|y|).
 $$
 A very similar argument yields the upper bound.
 \end{proof}

We remark that the bounds on the density $\bar \nu$ we have established above do not depend on the dimensions of $\X$ and $\Y$.  The lower bound is most relevant in \eqref{minPbNu} (since it is the lower bound on $\bar \nu$ that implies nestedness in the multi-to one-dimensional optimal transport problem, via Corollary \ref{cor: sufficient conditions for nestedness}); however, since the dimensions play no role, the result also applies to minimizers $\mu$ in problem \eqref{minPbMu}, in which case the upper bound can be used to prove nestedness via Corollary \ref{cor: nestedness from bounds on mu} (again with $n=1$).

Before developing these applications, we illustrate how the result above can be used to find an explicit bound in an example.

 \begin{example}\label{ex: required lower bound} Recall the quarter disk to arc problem from Example \ref{ex: disk to arc}: 
 $\mu$ is uniform on 
 $\X:=\{x_1,x_0>0: x_1^2+x_2^2 < 1 \} \subseteq \mathbb{R}^2$, so that $d\mu(x) =\frac{4}{\pi}dx$, $\Y=(0,\bar y)$, and the cost $c(x,y) = -x \cdot (\cos y ,\sin y) =-x_1\cos(y) -x_2\sin(y)$. We take $\mathcal{F}(\nu) =\int_Y\bar \nu \ln(\bar \nu)dy$ so that up to a non-vital constant $f'(\lambda) =\ln(\lambda)$.  We get that $K_v$ is the inverse of $z\mapsto \int_0^{\bar y} \exp{(z-v(y))}dy = \exp{(z)}\int_0^{\bar y}\exp{(-v(y))}dy$; that is, $K_{v}(1) =\ln\Big([\int_0^{\bar y}\exp{(-v(y))}dy]^{-1}\Big)$.  Noting that $M_c=1$, we have, for any minimizer $\nu$ of \eqref{minPbNu},
 \begin{equation}\label{eqn: lower bound on CN density}
 \begin{split}
\bar \nu(y) &\geq \exp\Bigg(\ln\Big([\int_0^{\bar y}\exp{(y)}dy]^{-1}\Big) -y\Bigg)\\ 
&=\frac{e^{- y }}{e^{\bar y} -1}.
\end{split}
\end{equation}

 \end{example}

\begin{remark}\label{rem: sharpness of bound}
	The bounds in Proposition \ref{prop: upper bound on density from conjestion}, and consequently Example \ref{ex: required lower bound}, are not sharp, since the inequality $-M_c  |y| \leq v(y) \leq M_c |y|$ arising from the Lipschitz constraint on $v$ is not sharp.  
	
	We illustrate this by considering the extreme case $\bar y =\frac{\pi}{2}$ in Example \ref{ex: required lower bound}.  In this case, uniform measure $\bar \nu (y) =\frac{2}{\pi}$ minimizes both the entropy $\mathcal{F}(\nu)=\int_Y\bar \nu \ln(\bar \nu)dy$ and the optimal transport $\nu\mapsto\mathcal{T}_c(\mu,\nu)$ distance to  $\mu$; it therefore minimizes \eqref{minPbNu}.  The lower bound provided by \eqref{eqn: lower bound on CN density} ranges from $\frac{e^{-\pi/2}}{e^{\pi/2} -1} \approx 0.055$ (for $y=\frac{\pi}{2}$) to  $\frac{1}{e^{\pi/2} -1} \approx 0.262$ (for $y=0$).
\end{remark}
\subsection{Minimizing congestion with a one dimensional target}\label{lbwi}
We now turn our attention to proving that minimizers in \eqref{minPbNu} are nested.  We begin by considering one dimensional targets and congestion (or internal energy) forms for $\mathcal F$, where the results in the last subsection can be combined with results in subsection  \ref{sect: bounds imply nestedness}.

We consider \eqref{minPbNu}, when the dimension of $\Y$ is $n=1$, and $\mathcal F[\nu] =\int_Yf(\bar \nu)dy$, and $f$ satisfies the conditions in subsection \ref{subsect: density bounds}.

Combining Corollary \ref{cor: sufficient conditions for nestedness} with the lower bound on the target density from  Proposition \ref{prop: upper bound on density from conjestion} yields the following.
\begin{theorem}
 \label{nestedness_congestion}  
Suppose that $\nu$ minimizes \eqref{minPbNu} 
over $P(\bar \Y)$, where $\Y =(0,\bar y)$ and $\X \subseteq \mathbb{R}^n$.  Then $(c,\mu,\nu)$ is nested provided 
$$
 \sup_{y_1 \in \Y,y_0\leq y\leq y_1} 	\frac{D^{min}_\mu(y_0,y_1,k(y_0))}{y_1-y_0}-(f')^{-1}(K_{-M_cy}(1)-M_cy)<0
$$
for all $y_0 \in \Y$.
\end{theorem}

In particular, we note the following consequence for our example matching the quarter circle to an arc with the bilinear cost.

 \begin{corollary}\label{cor: disk to circle full}
Suppose that $\nu$ minimizes \eqref{minPbNu} 
 	over $P(\bar \Y)$, where $\mu$ is uniform measure on the quarter disk $\X:=\{x_1,x_2>0: x_1^2+x_2^2 < 1 \} \subseteq \mathbb{R}^2$, $\Y=(0, \bar y)$, $F[\nu] = \int_Y\bar \nu\ln(\bar \nu)dy$ is the entropy and $c(x,y) =-x \cdot (\cos y, \sin y). $ 
 Then the model $(c,\mu,\nu)$ is nested provided $\bar y \leq  \ln \Big (\frac{1+\sqrt{(1+2\pi)}}{2} \Big)\approxeq 0.61$.
 \end{corollary}
 \begin{proof}
 	We only need to verify that the lower bound $\frac{e^{-\bar y }}{e^{\bar y} -1}$ on any solution $\nu$ obtained in Example \ref{ex: required lower bound},  is larger than the bound $\frac{2}{\pi}$, shown in Example \ref{ex: disk to arc} to guarantee nestedness.  This is an easy calculation.
 \end{proof}

\begin{remark}
	The sufficient condition for nestedness in Theorem \ref{nestedness_congestion} is not necessary, since the bound on $\bar \nu$ ensuring nestedness in Proposition \ref{prop: upper bound on density from conjestion} is not sharp in general (see Remark \ref{rem: sharpness of bound}).
	
	The upper bound on $\bar y$ in Corollary \ref{cor: disk to circle full} guaranteeing nestedness is not sharp, for two reasons:
	\begin{enumerate}
		\item The lack of sharpness in Theorem \ref{nestedness_congestion} described above.
		\item The bound on $D^{min}_\mu(y_0,y_1,k_0)$ in Example \ref{ex: disk to arc} is sharp only when $k_0 =0$, in which case $k_{max}(y_0,y_1,k_0)=0$, and the level curves $X_=(y_0,k_0)$ and $X_=(y_1,k_{max}(y_0,y_1,k_0))$ intersect at the origin. 
	\end{enumerate}  
It does not seem possible to get around the first issue above with the current techniques.  However, on the second issue, one can sometimes use local information on $k(y)$ to get improved control on $D^{min}_\mu(y_0,y_1,k(y_0))$.  In the context of Corollary \ref{cor: disk to circle full}, the local bounds from Example \ref{ex: required lower bound} on $\bar \nu(y)$ force the level sets to intersect higher up on the $x_2$ axis.  A more refined calculation, combining analytical and numerical methods,  shows that the solution in the Corollary is in fact nested as long as $\bar y$ is less than a certain upper bound, estimated numerically to be $\tilde y \approxeq 0.66$; details can be found in Appendix \ref{App}.

Even this bound is not sharp, however, because of issue 1 described above.  We note that when $\bar y =\frac{\pi}{2}$, and so the solution $\nu$ is constant as discussed in Remark \ref{rem: sharpness of bound}, the model is borderline nested in some sense: $k(y)=0$, so that all the $X_=(y,k(y))$ intersect at the origin, which lies in $\partial X$.  It is unclear to us whether solutions stay nested for all choices of $\bar y \in (\tilde y, \frac{\pi}{2}]$.
\end{remark}
 
 Once we have determined a priori that the solution must be nested, we can characterize it by a differential equation on the lower dimensional space (ie, an ordinary differential equation). 
 
 Note that, for a nested model, setting $k(y) =v'(y)$, one has, by \eqref{MAmulti2one},
 $$
 \bar \nu(y) =\int_{X_=(y,k(y))}\frac{D^2_{yy}c(x,y)-k'(y)}{|D^2_{xy}c(x,y)|}\bar \mu(x)d\mathcal H^{m-1}:=G(y,k(y),k'(y))
 $$

 and so differentiating the first order condition $v(y) + f'(\bar \nu(y))=C$ , we get the following second order differential equation for $k$\footnote{Generally speaking, the Kantorovich potential $v(y)$ is twice differentiable (so that $k'(y) =v''(y)$  is well defined) almost everywhere,   and the equation \eqref{MAmulti2one}, holds at these points \cite{PassM2one}. Equation \eqref{eqn: ode for CN} requires an extra derivative of $k$; in the present, nested setting, under certain conditions, one can actually deduce local continuous twice differentiability of $k$ (see Remark \ref{rem: bootstrapping}).  If these conditions do not hold, instead of using \eqref{eqn: ode for CN} one can  solve the first order condition for $\bar \nu(y)$ (with, say, $C=0$) and use \eqref{MAmulti2one} to write the second order equation for $v$: $G(y,v'(y),v''(y)) =(f')^{-1} (-v(y))$, which is well defined and holds almost everywhere.}:
\begin{equation}\label{eqn: ode for CN}
 k(y) +f''(G(y,k(y),k'(y)))\frac{d}{dy}[G(y,k(y),k'(y))]=0.
\end{equation}
 In addition,  we can derive boundary conditions for \eqref{eqn: ode for CN}: since we know the solution is nested, we have $\lim_{y \rightarrow 0^+}\mu(X_\geq(y,k(y))) =  \lim_{y \rightarrow 0^+}\nu(0,y) =0$, so that $k(0):=\lim_{y \rightarrow 0^+}k(y)$ exists, we have $\mu(X_\geq(0,k(0)))=0$ and $X_=(0,k(0))$ is tangent to $\partial X$.   This, and a similar argument as $y$ approaches $\bar y$ suggests that we impose the boundary conditions:
 
\begin{equation}\label{eqn: boundary conditions}
 k(0) =\max_{x \in \bar \X}\frac{\partial c}{\partial y}(x,0), \text{ }k(\bar y) =\min_{x \in \bar \X}\frac{\partial c}{\partial y}(x,\bar y).
\end{equation}
 Under the conditions of Corollary \ref{cor: sufficient conditions for nestedness}, any minimizer of \eqref{minPbNu} gives a solution to the above boundary value problem.  Conversely, it is not hard to see that given a solution $k(y)$ to \eqref{eqn: ode for CN} with boundary conditions \eqref{eqn: boundary conditions}, then $\bar \nu(y) =G(y,k(y),k'(y))$ is a minimizer provided that $v(y) =\int_0^yk(y)dy$ is $c$-concave.


 \begin{remark}[Bootstrapping to regularity]\label{rem: bootstrapping}	It is well known that the potential function $v(y)$ is Lipschitz \cite{McCann2001}; as the optimal $\nu$ has a lower bound, the equality $v(y) +f'(\bar \nu(y)) =C$ holds throughout $\Y$, and inverting yields that $\bar \nu$ is Lipschitz, with a constant determined by $f'$ and the upper and lower bounds on $\bar \nu$.
 	
 	Now, one can combine the optimality condition with the regularity theory developed in \cite{PassM2one} to bootstrap to higher regularity; nestedness together with $\bar \nu \in C^{0,1}(\Y)$ yields that $v$ is locally $C^{2,1}$, by Theorem 7.1 in \cite{PassM2one}. This in turn yields that $\bar \nu$ is in fact  $C^{2,1}$ (using $v(y) =C-f'(\bar \nu(y))$ again, and assuming sufficient smoothness of $f$); continuing in this way, we get that $\bar \nu(y), v(y)$ are locally $C^{r+1,1}$, provided $\bar \mu$ is $C^{r-1,1}$, $D_yc \in C^{r,1}$,   $\mathbf{n}_{\X} \in C^{r-2,1}$ ($\mathbf{n}_X$ denotes the outward unit normal to $\X$), and $\overline{X_=(y,k(y))}$ intersects $\partial X$ transversally.  The local norms are controlled by the quantities listed in Theorem 7.1 in \cite{PassM2one}.
 \end{remark}

 \subsection{Minimizing on the low dimensional marginal: interaction and potential terms}\label{subsect: interaction and potential}
In this section, we will consider the case where $\mathcal{F}$ consists of interaction and potential terms; that is, suppose that $\nu$ minimizes the following functional on $ \PP(\bar \Y)$:
\begin{equation}\label{minInt}
\nu \mapsto \mathcal T_c(\mu,\nu) +\int_Y V(y)d\nu(y) +\frac{1}{2}\int_Y\int_YW(y,z)d\nu(z)d\nu(y)
\end{equation}
where $\mu$ is a given probability measure on a set $\X \subseteq \mathbb{R}^m$, $\Y\subseteq \mathbb{R}^n$ with $m>n$ and the interaction term $W(y,z)=W(z,y)$ is symmetric.  We will denote by $F_{V,W}[\nu]$ the first variation of $\mathcal{F}$; that is,
\[F_{V,W}[\nu](y):=V(y) +\int_YW(y,z)d\nu(z).\]  In this case, we do not generally expect lower bounds on the density $\bar \nu$, and so the results from Section \ref{sect: bounds imply nestedness} tell us little about the structure of solutions.  However, under certain conditions, we will be able to use the optimality conditions directly to infer  generalized nestedness, as we will see below.

Assume throughout this section convexity of $\Y$ and uniform convexity of $y \mapsto c(x,y) +V(y) +W(z,y)$ throughout $\X \times \Y \times \Y$; that is, for all $x,y,z$, we have
\begin{equation}\label{eqn: uniform convexity assumption}
D^2_{yy}c(x,y) +D^2V(y) +D^2_{yy}W(z,y) \geq C>0.
\end{equation}
Also assume that for each $x \in \X$, $z \in \Y$ and $y \in \partial \Y$ that 
\begin{equation}\label{eqn: boundary max assumption}
[D_yc(x,y) +D V(y) +D _yW(y,z)] \cdot \mathbf{n}_{\Y}(y)\geq 0,
\end{equation}
where $\mathbf{n}_{\Y}$ is the outward unit normal to $\Y$.

\begin{theorem}\label{lemma: nestedness from interaction}
	Under the uniform convexity \eqref{eqn: uniform convexity assumption} and outward gradient \eqref{eqn: boundary max assumption} assumptions, $(c, \mu,\nu)$  satisfies the generalized nestedness condition for any minimizer $\nu$ of \eqref{minInt}.  Furthermore, if $c$, $V$ and $W$ are $C^k$ smooth (for any integer $k\geq 2$), then the optimal map between $\mu$ and $\nu$ is $C^{k-1}$.
\end{theorem}

\begin{proof} 
	For any solution $\nu$, we have the optimality condition
	
	$$
	v(y) +V(y) +\int_YW(z,y)d\nu(z)\geq 0
	$$
	with equality	$\nu$ almost everywhere, where $v(y)$ is the Kantorovich potential for the optimal transport problem \eqref{eqn: ot problem} between  $\nu$ and $\mu$.
	Integrating the uniform convexity assumption against $\nu(z)$, we get that 
	
	\begin{equation}\label{eqn: integrated uniform convexity}
	D^2_{yy}c(x,y) +D^2V(y) +\int_YD^2_{yy}W(z,y)d\nu(z) \geq C>0.
	\end{equation}
	Choose $\bar y \in spt(\nu) \cap \Y$ where $v$ is differentiable and $x \in X_=(\bar y, Dv(\bar y)) =\{x \in \X: D_yc(x,\bar y) =Dv(\bar y)\}$; we must show that $x \in \partial^cv(\bar y)$.  The (uniformly convex by \eqref{eqn: integrated uniform convexity}) function
	$$
	y \mapsto c(x,y)-v^c(x)+F_{V,W}[\nu](y):=c(x,y) -v^c(x) +V(y) +\int_YW(z,y)d\nu(z)
	$$
	has  a unique minimum $\tilde y$.  Now,  if that minimum is in the interior of $\Y$, the gradient vanishes there.  We claim that the gradient must vanish even if the minimum occurs on the boundary; in this case, the gradient must be a non positive multiple of the outward unit normal.  However, integrating  \eqref{eqn: boundary max assumption} against $\nu(z)$ implies 
	\begin{equation*}
	\begin{split}
	&\left[D_yc(x, \tilde y) +D V( \tilde y) +\int_YD_yW(x,\tilde y)d\nu(z) \right]\cdot \mathbf{n}_{\Y}(\tilde y)\\& =	\left[D_yc(x,\tilde y) +DF_{V,W}[\nu](\tilde y)\right]\cdot \mathbf{n}_{\Y}(\tilde y) \geq 0,
	\end{split}
	\end{equation*}
	which is only possible if $D_yc(x, \tilde y) +DF_{V,W}[\nu](\tilde y) =0$, establishing the claim.  Furthermore, by strict convexity, $\tilde y$ is the \emph{only} $y \in \bar \Y$ where $D_yc(x,  y) +DF_{V,W}[\nu]( y) =0$.
	
	 Noting the string of inequalities
	\begin{equation}\label{eqn: CN - OT inequalities}
	c(x,y) -v^c(x) +F_{V,W}[\nu](y) \geq v(y) +F_{V,W}[\nu](y) \geq 0,
	\end{equation}
	the minimum $\tilde y$ must coincide with the unique $y \in spt(\nu)$ such that equality holds, and we have $x \in \partial^cv(\tilde y)$.   
	
	We now show that $\tilde y =\bar y$.  To this end, we claim that  
	\begin{equation}\label{eqn: potential}
	Dv(\bar y) = -DF_{V,W}[\nu](\bar y).
	\end{equation}
	  Note that  as $\bar y \in spt(\nu)$, we have equality at $y=\bar y$ in the second inequality in \eqref{eqn: CN - OT inequalities}, and so if $\bar y$ is in the interior of $\Y$,  we get \eqref{eqn: potential} by minimality.  If $\bar y \in \partial \Y$, then $Dv(\bar y) + DF_{V,W}[\nu](\bar y)=\alpha \mathbf{n}_{\Y}(\bar y) $ with $\alpha \leq 0$.  We let $\tilde x \in \partial^cv(\bar y)$ (which is non-empty as $\bar y \in spt(\nu)$).  Then we have equality in \eqref{eqn: CN - OT inequalities} with $x=\tilde x$ and $y=\bar y$, and an identical argument to above (using \eqref{eqn: boundary max assumption} and minimality of the function) implies $D_yc(\tilde x,\bar y)  +D_yF_{V,W}[\nu](\bar y) =0$.  The non-negative function 
	  
	  $$
	  y \mapsto c(\tilde x,y) -v^c(\tilde x) +F_{V,W}[\nu](y) -\big( v(y) +F_{V,W}[\nu](y)\big)
	  $$
	  is then minimized at $\bar y$, and its gradient there must be a non-positive multiple of $\mathbf n_{\Y}(\bar y)$.  But this gradient is 
	  $$
	  D_yc(\tilde x,\bar y)  +DF_{V,W}[\nu](\bar y) -\big(Dv(\bar y) + DF_{V,W}[\nu](\bar y)\big)=0 -\alpha \mathbf{n}_{\Y}(\bar y) =-\alpha \mathbf{n}_{\Y}(\bar y)
	  $$
	  Thus, $\alpha \geq 0$, which (as $\alpha \leq 0$ as well) means $\alpha =0$.  This establishes \eqref{eqn: potential}.
	  
	   Therefore, we have
	$$
	D_yc(x,\bar y)+DF_{V,W}[\nu](\bar y) = D_yc(x,\bar y)-Dv(\bar y)=0
	$$
	and so $\bar y$ coincides with the minimum $\tilde y$  and $x \in \partial^cv(\bar y)$, as desired.
	
	We have now shown that $X_=(y,Dv(y)) =\partial^cv(y)$ for every $y \in \Y$ such that $v$ is differentiable.  To verify  generalized nestedness, we must show that this is $\nu$ almost every $y$.

	This follows by noting that the string of inequalities \eqref{eqn: CN - OT inequalities} implies that the semi-concave function $v$ is bounded from below by the smooth function $-F_{V,W}[\nu](y)$, with equality $\nu(y)$ almost everywhere.  At any point of equality, the gradient of $-F_{V,W}[\nu](y)$ is a subgradient for the everywhere superdifferentiable function $v$, and $v$ must therefore be differentiable there.
	
	The claimed regularity comes from the fact that the optimal map $T(x)$ coincides with the unique $y$ such that 
\begin{equation}\label{eqn: optimal map}
	D_{y}c(x,y) +DV(y) +\int_YD_{y}W(z,y)d\nu(z)=0,
\end{equation}
	combined with the implicit function theorem (noting that the left hand side of \eqref{eqn: optimal map} is the differential of a smooth uniformly convex function). 
\end{proof}

\begin{corollary}
	Under the assumptions of the lemma, the support  $spt(\nu) =T(spt(\mu))$ is connected if $spt(\mu)$ is.
\end{corollary}
\begin{proof}
	This follows from continuity of the optimal map $T$.
\end{proof}
 Generalized  nestedness of the solution and \cite{mccann2018optimal} now combine to imply the following result:
\begin{corollary}\label{cor: absolute continuity of nu}
	Assume that $c$ is twisted and non-degenerate, and adopt the assumptions of Theorem \ref{lemma: nestedness from interaction}.  Then the minimizer $\nu$ is absolutely continuous and its density satisfies the integral Monge-Ampere type equation \eqref{eqn: local unequal MA} almost everywhere.
\end{corollary}
\begin{proof}
Applying the chain rule to \eqref{eqn: optimal map} yields 
$$
DT(x) = -[D^2_{yy}c(x,T(x))+D^2_{yy}V(T(x))+\int_YD^2_{yy}W(z,T(x))d\nu(z)]^{-1}D^2_{yx}c(x,T(x)) ;
$$
and so non-degeneracy implies that $DT(x)$ is of full rank, so that $\nu =T_\#\mu$ is absolutely continuous.  Generalized nestedness and Theorem 1 in \cite{mccann2018optimal} then yield  equation \eqref{eqn: local unequal MA}.

\end{proof}

\begin{remark}\label{rem: smoothness of nu}
The integro Monge-Ampere operator  appearing in \eqref{eqn: local unequal MA} has regularity controlled by a variety of quantities depending on $c$, $\X$, $\Y$ and $\mu$ (see Theorem 11 in \cite{mccann2018optimal}).  Since the potential $v(y) = -V(y) -\int_YW(y,z)d\nu(z)$ is as smooth as $V$ and $W$ on the support of $\nu$, \eqref{eqn: local unequal MA} then yields regularity estimates on $\bar \nu(y)$.  
\end{remark}

Finally, we note that on the support of $\nu$, we can eliminate $v$ from \eqref{eqn: local unequal MA} to obtain the following partial differential equation for $\bar \nu(y)$:
\begin{equation}\label{eqn: pde for interaction}
\bar \nu(y)=G(y, -DF_{V,W}[ \nu](y), -D^2(F_{V,W}[ \nu])(y))
\end{equation}
where $F_{V,W}[\nu](y) =V(y) +\int_Y W(y,z)d\nu(z)$ is linear in $\nu$ and $$G(y,p,Q) =\int_{X_=(y,p)} \frac{\det[D_{yy}^2c(x,y)-Q]}{\sqrt{\det[D_{yx}^2c(x,y)D_{xy}^2c(x,y)]}}$$ is the integro Monge-Ampere operator from \cite{mccann2018optimal}. 

Two complications, absent in the congestion case, arise here: first, the operator $F_{V,W}$ depends \emph{non-locally} on $\nu$, and so the PDE \eqref{eqn: pde for interaction} is non-local, even though the model satisfies the generalized nestedness condition, which eliminates potential non-locality arising from the integro Monge-Ampere operator $G$ as in \cite{mccann2018optimal}.  Second, we do not know the  support of $\nu$ in advance, only that it is a connected subset of $\Y$; we therefore cannot impose boundary conditions.  These issues are not artefacts of the unequal dimensional setting; they arise in equal dimensional problems as well.  Since they seem to make solving the problem via the PDE approach challenging, they serve as good motivation for an iteration scheme, adapted from Blanchet-Carlier \cite{blanchet2014remarks} and developed below.
\subsubsection{A fixed point characterization}
Noting that by differentiating with respect to $y$ the optimality condition for \eqref{minInt} we obtain
\begin{equation}
\label{OPtCond}
 D_yc(x,y)+D F_{V,W}[\nu](y)=0 ,
 \end{equation}
 we denote by $B_\nu:\X\to\Y$ the map such that  
 \begin{equation}\label{eqn: iteration map} 
 D_yc(x,B_\nu(x))+D F_{V,W}[\nu](B_\nu(x))=0,
 \end{equation}
which is well defined under conditions \eqref{eqn: boundary max assumption} and \eqref{eqn: uniform convexity assumption}.
Then, the scheme introduced in \cite{blanchet2014remarks} consists in iterating the application defined as  
\begin{equation}
\label{It}
\mathcal B(\nu):=(B_\nu)_\sharp\mu. 
\end{equation}
The following Theorem establishes the existence of a unique fixed point $\nu^\star$ of \eqref{It} which is a solution to \eqref{minInt}.
\begin{theorem}{(The best reply iteration scheme-unequal dimensional case)}
\label{bestReply}
Let $\mu\in\PP(\X)$ and the application $\mathcal B:\PP(\Y)\rightarrow\PP(\Y)$ 
defined in \eqref{It}.
Assume that the transport cost $c(x,y)$ is uniformly convex in $y$, that is $D_{yy}^2 c\geq \eta\id$ with $\eta>0$, $D^2_{xy}c$ has maximal rank and 
$F_{V,W}[\nu]$ satisfies the following hypothesis 
\begin{align}
\label{Hyp1bis}
D^2F_{V,W}[\nu]\geq\lambda\id\;in\; \Y, \lambda>0;&\\
\label{Hyp2bis}
\mathcal H^{m-n}(B^{-1}_\nu(y))\leq M \;\forall y\in \Y, M\in\R;&\\
\label{Hyp4}
JB_{\nu}\geq k>0\;in\;\X;&\\ 
\label{Hyp3bis}
\int_{\Y}  |D F_{V,W}[\nu_1]-D F_{V,W}[\nu_0]|dy\leq C\mathcal{W}_1(\nu_1,\nu_0)&
\end{align}
where $JB_{\nu}$ is the $n-$dimensional Jacobian of $B_{\nu}$
Moreover, let $\mu\in\PP(\X)$ absolutely continuous with respect to Lebesgue and such that $||\mu||_\infty M C<k(\eta+\lambda) $ .\\
Then $\mathcal B$ is a contraction of $(\PP(\Y),\mathcal{W}_1)$, where we denote by $\mathcal{W}_1$ the 1-Wasserstein distance (namely the Optimal Transport problem with the Monge cost) and the unique fixed point $\nu^\star$ is solution to \eqref{minInt}.
\end{theorem}
\begin{proof}

Given $\nu_{0},\nu_1\in\PP(\Y)$ and $y_i(x):=B_{\nu_i}(x)$, it follows from uniform convexity in $y$ of $c$ that
\[ \big( D_y c(x,y_1)-D_y c(x,y_0)\big)\cdot (y_1-y_0) \geq \eta |y_1-y_0|^2. \]
Then, by using the definition of $y_i$ we have
\begin{equation*}
\begin{split}
&\eta|y_1-y_0|^2\leq (y_1-y_0)\cdot\Big( D_y c(x,y_1)-D_y c(x,y_0)\Big)\\ &= (y_1-y_0)\cdot\Big( DF_{V,W}[\nu_0](y_0)-DF_{V,W}[\nu_1](y_1)\Big) \\
&=  (y_1-y_0)\cdot\Big ( DF_{V,W}[\nu_0](y_0)-DF_{V,W}[\nu_0](y_1) + DF_{V,W}[\nu_0](y_1)-DF_{V,W}[\nu_1](y_1)\Big).
\end{split}
\end{equation*}
Applying \eqref{Hyp1bis} we have
\begin{equation*}
\eta|y_1-y_0|^2\leq  \Big(-\lambda |y_1-y_0|^2+|y_1-y_0||DF_{V,W}[\nu_0](y_1)-DF_{V,W}[\nu_1](y_1)|\Big)
\end{equation*}
 and so
 \begin{equation*}
|y_1-y_0|\leq \frac{1}{\eta+\lambda}| DF_{V,W}[\nu_0](y_1)-DF_{V,W}[\nu_1](y_1)|.
\end{equation*}
Thus, now
\begin{align*}
&\mathcal{W}_1(\mathcal B(\nu_1),\mathcal B(\nu_0))\leq \int_{\X}  |y_1(x)-y_0(x)|d\mu(x)\\
& \leq \frac{1}{\eta+\lambda}\int_{\Y}| DF_{V,W}[\nu_0](y_1)-DF_{V,W}[\nu_1](y_1)|d\mathcal B(\nu_1).
\end{align*}
$\mathcal B(\nu_1)$ has a density with respect to Lebesgue given by the co-area formula
\begin{equation}
\label{CoA}
  \mathcal B(\nu_1)(y)=\int_{B^{-1}_{\nu_1}(y)} \dfrac{\bar \mu(x)}{JB_{\nu_1}}d\mathcal{H}^{m-n}(x), 
 \end{equation} 
where $JB_{\nu_1}$ denotes the $n-$dimensional Jacobian of $B_{\nu_1}$.
Notice that our uniform convexity assumptions, together with the implicit functions theorem imply differentiability of $B_{\nu}$, and that differentiating \eqref{eqn: iteration map}
 we have
\[  \Big (  D^2F_{V,W}[\nu](B_{\nu}(x))+D^2_{yy}c(x,B_{\nu}(x))  \Big)DB_{\nu}(x)=-D^2_{xy}c(x,B_{\nu}(x)) \]
and, since  the right hand side has rank $n$ (by hypothesis), we conclude both factors on the left must have rank $n$. This actually implies that $JB_{\nu_1}>0$ and the co-area formula holds.  
Now by \eqref{CoA} and \cref{Hyp2bis,Hyp3bis,Hyp4} we obtain
\begin{align*}
&\mathcal{W}_1(\mathcal B(\nu_1),\mathcal B(\nu_0))\\&\leq \frac{1}{\eta+\lambda}\int_{\Y}| DF_{V,W}[\nu_0](y_1)-DF_{V,W}[\nu_1](y_1)|\Big ( \int_{B^{-1}_{\nu}(y)}\dfrac{\bar\mu(x)}{JB_\nu}d\mathcal H^{m-n}(x)   \Big)dy\\
&\leq \frac{||\mu||_\infty}{\eta+\lambda}\int_{\Y}| DF_{V,W}[\nu_0](y)-DF_{V,W}[\nu_1](y)|\Big ( \int_{B^{-1}_{\nu}(y)}\dfrac{1}{JB_\nu}d\mathcal H^{m-n}(x)   \Big)dy\\
& \leq \frac{||\mu||_\infty M C}{k(\eta+\lambda)} \mathcal{W}_1(\nu_1,\nu_0).
\end{align*}
Since  $||\mu||_\infty M C< k(\eta+\lambda) $, we can conclude the proof by Banach's fixed point theorem.
\end{proof}
\begin{remark}
One can get rid of hypothesis \cref{Hyp4} by noticing that the Jacobian of $B_{\nu}$ depends on other quantities: $D^2F_{V,W}[\nu](y)$, $D^2_{yy}c(x,y)$ and $D^2_{xy}c(x,y)$.
\end{remark}
\begin{remark}[The equal dimensional case]
When $m=n$, the above proposition is an extension of \cite[Theorem 5.1]{blanchet2014remarks} to the case in which a general cost function is involved.\\
First of all let us remark if $c(x,y)$ is double twisted, that is  $x\mapsto D_yc(x,y)$ is injective, then from $D_yc(x,y)=p$ one can deduce $x$ uniquely from  $y$ and $p$, in which case we can write $x=c\operatorname{-exp}_y(p):=D_yc(\cdot,y)^{-1}(p)$.
Thus by using the optimality condition \eqref{OPtCond} and the  injectivity of  $x\mapsto D_yc(x,y)$ we have 
$$ x=c\operatorname{-exp}_y(-DF_{V,W}[\nu])(y). $$
It is now clear that the map $B_\nu:\X\rightarrow\Y$ we have defined above is given by
\[ B_\nu(x):=(c\operatorname{-exp}_y(-DF_{V,W}[\nu]))^{-1}(x).\]
\eqref{CoA} can now be  replaced by the change of variable  formula
\[ \mathcal B(\nu_1)(y)=\mu(B_{\nu_1}^{-1}(y))\det(B_{\nu_1}^{-1}(y)) \]
and this implies that \cref{Hyp2bis} holds with $M=1$.
In the case of quadratic cost this coincides with the hypothesis in  \cite[Theorem 5.1]{blanchet2014remarks}.
We highlight  the map $B_\nu(x)$ is actually not explicit or simple to compute. However in the special case in which $c(x,y)=h(x-y)$, with $h$ is strictly convex, the map $B_\nu(x)$ takes the form
$$B_\nu(x)=(\id+Dh^{-1}(-DF_{V,W}[\nu]))^{-1}(x).  $$ 
So far we have assumed that the cost is double twisted, but we can avoid this assumption and notice that \eqref{CoA} still holds. In this case we have that
\[  \mathcal B(\nu_1)(y)=\int_{B^{-1}_{\nu_1}(y)} \dfrac{\mu(x)}{JB_{\nu_1}}d\mathcal{H}^{0}(x),  \]
and, since $\mathcal H^{0}$ is simply the counting measure (under the non-degeneracy condition), this implies that \eqref{Hyp2bis} can be interpreted as a bound on the number of points in the pre-image of $B_{\nu_1}$.
\end{remark}
\subsection{Minimizing over the high dimensional marginal with one dimensional target}
We now consider problems where the high dimensional measure $\mu$ is allowed to vary, restricting to the one dimensional target setting, $n=1$; first we study the case where the target measure $\nu$ is fixed.
\subsubsection{Minimizations with a fixed target measure}\label{subsect: min high dim}
Consider fixing $\nu$ and minimizing $\mu \mapsto \mathcal T_c(\mu, \nu) +\mathcal{G}(\mu)$, where $\mathcal{G}(\mu) = \int_Xg(\bar \mu(x))dx$ is a congestion type functional, with $g$ satisfying the conditions on $f$ in subsection \ref{subsect: density bounds}, 
the domains $\X \subseteq \mathbb{R}^m$ and $\Y \subseteq \mathbb{R}$ is one dimensional.  Combining Corollary \ref{cor: nestedness from bounds on mu} and Proposition \ref{prop: upper bound on density from conjestion}, we immediately obtain the following.

\begin{theorem}
\label{high_dimensional_congestion}
	Assume that for all $y_0 <y <y_1$, 
	and \\ $x \in X_\geq(y_1,k_{max}(y_0,y_1,k(y_0)))\setminus X_\geq(y_0,k(y_0))$ we have
	$$
 (g')^{-1}(K_{M_c|x |}(1) + M_c|x|)< \frac{\bar\nu(y)(y_1-y_0)}{D^{min}_{vol}(y_0,y_1,k(y_0))}.
	$$
	where $M_c$ is a Lipschitz constant for $ y \mapsto c(x,y)$.
	Then the model $(c, \mu, \nu)$ is nested for any minimizer $\mu$.
\end{theorem}

Note that the equality $\bar \mu(x) = (g')^{-1}(C-u(x))$, and the general fact that the potential $u(x)$ is Lipschitz implies that the optimal marginal $\bar \mu(x)$ is Lipschitz as well.  This allows one to use Theorem 7.1 in \cite{PassM2one} to obtain interior $C^{2,1}$ estimates on $v= u^c$.  It is not clear to us whether this can be bootstrapped to obtain higher regularity.

\subsubsection{Double minimizations}
Consider now the problem where neither measure is fixed, and where  $\mathcal{G}(\mu)=\int_Xg(\bar\mu(x))dx$ and $F_{V,W}$  have the forms in subsections \ref{subsect: density bounds} and \ref{subsect: interaction and potential}, respectively.  That is, consider the minimization problem
\begin{equation}\label{eqn: double min}
\inf\left\{\mathcal T_c(\mu,\nu)+ \mathcal G(\mu) +\int_Y F_{V,W}[\nu]d\nu(y)\;:\;(\mu,\nu)\in\PP(\bar \X)\times\PP(\bar \Y)\right\}.
\end{equation}
The results established above can be used to prove the following.

\begin{theorem}
\label{double_minimization}
Adopt the assumptions on $\X, \Y, c, V$ and $W$ from the previous section, and assume that $g$ satisfies the conditions in subsection \ref{subsect: density bounds}.  Then, whenever  $(\mu,\nu)$ minimizes \eqref{eqn: double min},
\begin{enumerate}
	\item $(c, \mu, \nu)$ is nested.
	\item $\mu$ is absolutely continuous with an everywhere positive density.
	\item The optimal map $T$ is two degrees less smooth than $c,V$ and $W$, while the Kantorovich potential $u(x)$ and  density $\bar \mu(x)$ are  one degree less smooth than $c,V$ and $W$.  
	\item $\nu$ is absolutely continuous.
\end{enumerate}	
\end{theorem}
\begin{proof}
  Absolute continuity of $\mu$ follows immediately from the conditions on $g$, while Proposition \ref{prop: upper bound on density from conjestion} ensures 
   that the density is everywhere positive.   Nestedness and the smoothness of the optimal map follow directly from Theorem \ref{lemma: nestedness from interaction}.  The first order condition
$$
Du(x) = D_xc(x,T(x))
$$
then means that $u$ is one degree less smooth than $c, V$ and $W$, as desired.  Via the equality $u(x) +g'(\bar \mu(x))=C$, this means that $u$ is equally smooth. 

Corollary \ref{cor: absolute continuity of nu} 
then applies to yield regularity of $\nu$.  
\end{proof}
As noted in Remark \ref{rem: smoothness of nu}, higher regularity of $\bar \nu$ can be obtained, depending on $c$, $\X$, $\Y$, $V$ and $W$.

\section{Hedonic pricing problems}
In this section, we study the hedonic pricing problem found in \cite{Ekeland05} and \cite{ChiapporiMcCannNesheim10}; economically, this problem involves matching distributions $\mu_1$ and $\mu_2$ of buyers and sellers on spaces $X_1 \subseteq \mathbb{R}^{m_1}$ and $X_2 \subseteq \mathbb{R}^{m_2}$ (both assumed bounded and open), with $m_1,m_2 \geq 1$, according to their preferences for goods in a space $\Y$ (which we will assume is one dimensional).  Mathematically, this amounts to taking $\mathcal{F}(\nu)$ to be the optimal transport distance to another fixed measure in \eqref{minPbNu} \cite{Ekeland05}\cite{ChiapporiMcCannNesheim10}.    We therefore seek to minimize: 

\begin{equation}\label{eqn: hedonic problem}
\min_{\nu \in P(\bar \Y)}\mathcal{T}_{c_1}(\mu_1,\nu)+\mathcal{T}_{c_2}(\mu_2,\nu),
\end{equation}
where the $\mu_i \in P(X_i)$ are absolutely continuous probability measures on the $X_i$ and $\Y \subseteq \mathbb{R}$.  Each $\mathcal{T}_{c_i}$ represents the  optimal transport distance \eqref{eqn: ot problem} between $\mu_i$ and $\nu$ with respect to a $C^2$, non-degenerate cost function $c_i(x_i,y)$.  We attempt to construct a solution by adapting the construction for the straight optimal transport problem in \cite{PassM2one} as follows:

Fix $y$.  For each $M \in [0,1]$, choose the unique $k_i=k_i(y,M)$ such that $\mu_i(X_\geq^i(y,k_i)) =M$, where 
$$
X_\geq^i(y,k_i) :=\{x_i \in X_i: D_yc_i(x_i,y) \geq k_i\};
$$
we adopt similar notation for the level sets $X_=^i(y,k_i)$.
Now consider the function $M \mapsto k_1(y,M)+k_2(y,M)$. The map is continuous and strictly decreasing.  

\begin{lemma}
	Assume $y$ is in the interior of $\Y$ and $y \in argmin (c_1(x_1,y) +c_2(x_2,y))$ for some $(x_1,x_2) \in X_1 \times X_2$. Then  the mapping $M \mapsto k_1(y,M)+k_2(y,M)$ has a unique $0$. 
\end{lemma}
\begin{proof}
	The argmin condition implies that $0 =D_y(c_1(x_1,y) +c_2(x_2,y))$, so that $0$ is in the range of the mapping $(x_1,x_2) \mapsto D_y(c_1(x_1,y) +c_2(x_2,y))$.\\
	
Now, note that $k_i(y,0) =\max_{x_i \in \bar \X}D_yc_i(x_i,y)$ while $k_i(y,1) =\min_{x_i \in \bar \X}D_yc_i(x_i,y)$.  The continuous mapping $M \mapsto k_1(y,M)+k_2(y,M)$ then maps the interval $[0,1]$ onto the interval $[\underline \alpha, \overline \alpha]$, where $\underline \alpha =\min_{x_1 \in \bar \X}D_yc_1(x_1,y) +\min_{x_2 \in \bar \X}D_yc_2(x_2,y)$ and $\overline \alpha =\max_{x_1 \in \bar \X}D_yc_1(x_1,y) +\max_{x_2 \in \bar \X}D_yc_2(x_2,y)$.

	 This range $[\underline \alpha, \overline \alpha]$ coincides with the range of $(x_1,x_2) \mapsto D_y(c_1(x_1,y) +c_2(x_2,y))$; since $0$ is in the latter, it must also be in the former.  That is, there is an $M$ such that $k_1(y,M)+k_2(y,M)=0$.  By strict monotonicty, this $M$ is unique, completing the proof. 
\end{proof}
Denote the zero from the preceding Lemma by $M(y)$.   We say the problem \eqref{eqn: hedonic problem} is \emph{hedonically nested} if 
\begin{equation}\label{eqn: hedonic nestedness}
X^i_\geq (y,k_i(y,M(y))) \subseteq X^i_> (\bar y,k_i(\bar y,M(\bar y))) 
\end{equation}
for $i=1,2$, whenever $y,\bar y \in \Y$ with $y < \bar y$. As we show below, this is equivalent to $M$ being the cumulative distribution function of a probability measure $\nu$ \textit{and}  $(c_i,\mu_i,\nu)$ being 
nested for $i=1$ and $2$.
\begin{theorem}
\label{hedonical_nestedness}
	The problem is hedonically nested if and only if $M(y)$ is the  cumulative distribution function of some probability measure $\nu$ and $(c_i,\mu_i,\nu)$ 
	is nested for $i=1,2$.  In this case, $\nu$ is optimal in \eqref{eqn: hedonic problem}.
\end{theorem}
\begin{proof}
	The hedonic nestedness condition \eqref{eqn: hedonic nestedness} for either $i=1$ or $2$ implies that $M(y) = \mu_i(X^i_\geq (y,k_i(y,M(y)))) \leq \mu_i(X^i_> (\bar y,k_i(\bar y,M(\bar y))) ) =M(\bar y)$.  Therefore
	$y\mapsto M(y)$ is monotone increasing, so that $M$ is indeed the cdf of a probability measure $\nu$; the condition also implies that the models $(c_i,\mu_i,\nu)$ both  satisfy \eqref{eqn: 1-d nestedness}. Conversesly, if $M(y)$ is the cumulative distribution function of $\nu$, it follows immediately that $k_i(y, M(y))$ coincides with the $k(y)$ defined by \eqref{eqn: def of k} and used in the definition of nestedness; the nestedness condition \eqref{eqn: 1-d nestedness} then implies that the hedonic nestednss condition \eqref{eqn: hedonic nestedness} holds.

	Turning to the second assertion, since each $(c_i,\mu_i,\nu)$ satisfies \eqref{eqn: 1-d nestedness}, \cite{PassM2one} implies that the mapping sending each $x_i$ in $X_=(y,k_i(y,M(y)))$ to $y$ is the optimal map between $\mu_i$ and $\nu$ and we have $v_i'(y) =k_i(y,M(y))$, where $v_i$ is the Kantorovich potential.  We then have by construction
	$$
	v_1(y)+v_2(y)=C
	$$
	for all $y$; this is exactly the optimally condition for \eqref{eqn: hedonic problem} (see \cite{Ekeland05}) and implies optimality of $\nu$.

\end{proof}
\begin{remark}
	Note that a similar construction will hold for the matching for teams problem from \cite{carlier2010matching}, where one minimizes $\nu \mapsto \sum_{i=1}^N\mathcal T_{c_i}(\mu_i,\nu)$ over probability measures on $\Y \subseteq \mathbb{R}$.
Notice that in the case in which $c_i(x,y)=\lambda_i|x-y|^2$ with $\lambda_i\geq0$ and $\sum_i\lambda_i=1$,  this problem 
 can be read as an \textit{unequal dimensional} version of the Wasserstein barycenters problem introduced in \cite{Carlier_wasserstein_barycenter}.
\end{remark}
We next note that the nestedness of either one of the $(c_i, \mu_i, \nu)$ (implied, for instance, by hedonic nesting) implies that the solution $\nu$ vanishes at the boundary; in economic terms, this means that neither the lowest nor highest quality goods are exchanged in equilibrium.
\begin{corollary}
	Suppose $\Y =(\underline y,\overline y)$ is an interval. Assume $(c_i, \mu_i, \nu)$ is nested for the optimal $\nu$, for either $i=1$ or $2$, and that the density $\bar \mu_i$ is bounded. 
	Set $\underline k = \max_{x \in \bar \X}D_yc_i(x,\underline y)$ and $\overline k = \min_{x \in \bar \X} D_yc_i(x,\overline y)$.

	 If $\lim_{k \rightarrow \underline k^-}\mathcal H^{m_i-1}(X^i_=(\underline y,k))=0$,  
	 then the optimal density is zero at $\underline y$, $\bar\nu(\underline y) =0$. 
	 
	 Similarly, if $\lim_{k \rightarrow \overline k^+}\mathcal H^{m_i-1}(X^i_=(\overline y,k))=0$, 
	 then the optimal density is zero at $\overline y$, $\bar\nu(\overline y) =0$.
\end{corollary}	
\begin{proof}
	Without loss of generality, we take $i=1$.  We'll prove the claimed result about $\underline y$; the $\overline y$ argument is identical.
	If $\underline y$ is outside the support of $\nu$, the result follows immediately.  If not, 
	note that as $v_1( y)+ v_2( y)\geq 0$ with equality $\nu$ almost everywhere, and the Kantorovich potentials $v_1$ and $v_2$ are semi-concave, then a standard argument implies that each of $v_1$ and $v_2$ are twice differentiable at any point where equality holds, and, at such points, $v''_1(y) +v''_2(y) \geq 0$. It then follows, using the standard fact that the $c$-concave function $v_2$ satisfies $v_2''(y) \leq D^2_{yy}c_2(x,y)$ for any  $x \in \partial ^{c_2}v_2(y)$,  that $v_1''(y) \geq -v''_2(y) \geq -D^2_{yy}c_2(x,y)$, for $x \in \partial ^{c_2}v_2(y)$.  
Since the continuous function $D^2_{yy}c_2$ is bounded from above on $\bar \X \times \bar \Y$, this means that 
	$$
	v''_1 \geq C > -\infty
	$$
	$\nu$ almost everywhere.  Since \eqref{MAmulti2one} also holds $\nu$ almost everywhere, we have
	$$
	\bar \nu(y) \leq \bar C \mathcal H^{m_1-1}(X_=(y,k_1(y))),
	$$
	for an appropriate constant $\bar C$.  Noting that $k_1(y) \rightarrow \underline k$ as $y \rightarrow \underline y^+$ then yields the result for $\underline y$ in the support of $\nu$.  
	
\end{proof}
The condition $\lim_{k \rightarrow \underline k^+}\mathcal H^{m_i-1}(X^i_=(\underline y,k))=0$ heuristically means that the first level set of $x \mapsto (c_i)_y$ that intersects $\bar X_i$ does so in a lower dimensional way.  Since this level curve is tangent to $\partial X_i$, this is generically true.  Note that when $c_i(x,y) =x \cdot \alpha(y)$ for some curve $\alpha: \mathbb{R}\to \mathbb{R}^m$, the level sets are hyperplanes and so strict convexity of $\X$, or in fact the slightly weaker condition that $\partial \X$ has no $m-1$ dimensional facets introduced in \cite{FigalliKimMcCann2011}, suffices.

Below, we will present general, differential conditions on the functions $c_i$ and measures $\mu_i$ which guarantee hedonic nestedness.  First however, we present an example illustrating how the above procedure can be used to construct a solution.

\begin{example}
	
 Match a uniform distribution of consumers on $\X_1 =(0,1)^2$ with goods on $\Y= (-3,3)$ and costs $c_1(x_1,y) =c_1(x_1^1,x_1^2,y)=x_1\cdot ( y^2/2, -y)=-x_1^2y +x_1^1y^2/2$ and sellers, uniformly distributed on $\X_2=(0,1)$ with preferences $c_2(x_2,y)=-x_2y+y^2/2$.

We compute $(c_1)_y=-x_1^2+x_1^1y$; the sets $X^1_=(y,k_1)$ are line segments $x_1^2=yx_1^1-k_1$ with slope $y$ and intercept $-k_1$.  When $y> 0$ and $k_1<0$, the  measure of the super-level set $X^1_\geq(y,k_1)$ is $ y/2-k_1$  if $y-k_1 \leq 1$  and $1-(1+k_1)^2/2y$ if and $y-k_1 \geq 1$.  Equating to $M$ and inverting yields $k_1(y,M)=y/2-M$ and $k_1(y,M)= \sqrt{(1-M)2y}-1$ in these two regions.  (It will turn out that the level curves with either $y\leq0$ or $k_1>0$ will not be involved in the solution.)

On the other hand, $(c_2)_y=y-x_2$ and so the measure of $X^2_\geq(y,k_2)$ is $y-k_2$, leading to $k_2(y,M) =y-M$.  

So for small $M$, the equation $k_1(y,M) +k_2(y,M)=0$ yields $M=3y/4$.  When $y=\frac{4}{5}$, and so $M=\frac{3}{5}$ and $k_1 =\frac{1}{5}$ we transition to the other form of $k_1$, and thus must solve $\sqrt{(1-M)2y}-1+ y-M=0$, which leads to
$$
M=\sqrt{4y-y^2} -1.
$$
Note that we get $M=1$ when $y=2$.  It remains to show that the hedonic nesting conditions \eqref{eqn: hedonic nestedness} hold.  This is trivial for $i=2$, since $X_2$ is one dimensional.  For $i=1$, it suffices to show that the intercept $-k_1$ of the level curves $x_1^2=yx_1^1-k_1$ is increasing in $y$ (note that the slope is clearly monotone increasing). For $y \leq \frac{4}{5}$, we have $-k_1(y,M(y)) =M(y)-y/2 =3y/4-y/2=y/4$, which is clearly monotone.  For $y \geq \frac{4}{5}$, we have $-k_1(y,M(y)) =k_2(y,M(y)) =-\sqrt{4y-y^2} +1 +y$.  The derivative of this function is:

\begin{eqnarray*}
	\frac{d}{dy}k_1(y,M(y))&=&-\frac{2-y}{\sqrt{4y-y^2}} +1\\
	&=&\frac{\sqrt{4y-y^2}-2+y}{\sqrt{4y-y^2}}\\
	&=&\frac{4y-y^2-(2-y)^2}{\sqrt{4y-y^2}(\sqrt{4y-y^2}+(2-y))}\\
	&=&\frac{-2((2-y)^2-2)}{\sqrt{4y-y^2}(\sqrt{4y-y^2}+(2-y))}
\end{eqnarray*}
To ensure this is non-negative, it suffices to show $(2-y)^2-2<0$ on $[4/5,2]$.  That is, $-\sqrt2<2-y <\sqrt2$, or $2-\sqrt2<y <2+\sqrt2$, which is clearly true.
\end{example}
\subsection{Differential conditions ensuring hedonic nestedness}
The result below identifies differential conditions on the cost functions $c_i$ and marginals $\mu_i$ under which the model is hedonically nested.

\begin{lemma}
	Assume that both $c_i$'s are uniformly convex with respect to $y$.  Also assume that for each fixed $\bar x_1 \in X_1$,  $\bar x_2 \in X_2$, and $y\in \Y$ such that  $k_1+k_2=0$ and $\mu_1(X^1_\leq (y,k_1)) = \mu_2(X^2_\leq (y,k_2))$, where $k_i =(c_i)_y(\bar x_i,y)$ for $i=1$ and $2$, we have
	
	\begin{equation}
	\begin{split}
		\int_{X^1_=(y, k_1)}&\left[\frac{D^2_{yy}c_1}{|D^2_{x_1y}c_1|}(x_1,y)\right]\bar \mu_1(x_1)d\mathcal H^{m_1-1}(x_1)\\
		&-\int_{X^2_=(y, k_2)}\left[\frac{D^2_{yy}c_2}{|D^2_{x_2y}c_2|}(x_2,y)\right]\bar \mu_2(x_2)d\mathcal H^{m_2-1}(x_2)\\
		&-D^2_{yy}c_1(\bar x_1, y)\int_{X^1_=(y, k_1)}\left[\frac{1}{|D^2_{x_1y}c_1|}(x_1,y)\right]\bar \mu_1(x_1)d\mathcal H^{m_1-1}(x_1)\\
		&-D^2_{yy}c_1(\bar x_1, y)\int_{X^2_=(y, k_2)}\left[\frac{1}{|D^2_{x_2y}c_2|}(x_2,y)\right]\bar \mu_2(x_2)d\mathcal H^{m_2-1}(x_2)\\
		&<0.\label{eqn: condition for nested hedonic equilibrium}
\end{split}
\end{equation}
Then condition \eqref{eqn: hedonic nestedness} holds for $i=1$.

\end{lemma}
Therefore, if this condition and its analogue with the roles of $i=1$ and $i=2$ reversed both hold, the model is hedonically nested.
\begin{proof}
	We will show that, for the $k_1(y) =k_1(y,M(y))$ constructed above, we have  $k_1'(y) -D^2_{yy}c_1(x_1,y) <0$ throughout $X_=^1(y,k_1(y))$.  This will imply \eqref{eqn: hedonic nestedness}, as in \cite[Corollary 5.3]{PassM2one}.
	The equation $\mu_1(X^1_{\geq}(y,k_1(y,M)))=M$ implicitly defines the function $k_1$; differentiating and using the formulas in \cite{PassM2one} for the derivatives of $\mu_1(X^1_{\geq}(y,k_1))$ with respect to $y$ and $k_1$ yields
	$$
	\frac{\partial k_1}{\partial M} =-\dfrac{1}{\displaystyle\int_{X^1_=(y, k_1)}\left[\dfrac{1}{|D^2_{x_1y}c_1|}(x_1,y)\right]\bar \mu_1(x_1)d\mathcal H^{m_1-1}(x_1)}
	$$
	and
	$$
	\frac{\partial k_1}{\partial y} =\dfrac{\displaystyle\int_{X^1_=(y, k_1)}\left[\dfrac{D^2_{yy}c_1}{|D^2_{x_1y}c_1|}(x_1,y)\right]\bar \mu_1(x_1)d\mathcal H^{m_1-1}(x_1)}{\displaystyle\int_{X^1_=(y, k_1)}\left[\dfrac{1}{|D^2_{x_1y}c_1|}(x_1,y)\right]\bar \mu_1(x_1)d\mathcal H^{m_1-1}(x_1)}.
	$$
	
	Now, the equation $k_1(y,M) +k_2(y,M) =0$ implicitly defines $M(y)$; differentiating, we have
	$$
	\frac{\partial k_2}{\partial y} +\frac{\partial k_1}{\partial y} + \Big[\frac{\partial k_1}{\partial M}+\frac{\partial k_2}{\partial M}\Big]M'(y)=0
	$$
	or $M'(y) =-\dfrac{\frac{\partial k_2}{\partial y} +\frac{\partial k_1}{\partial y}}{\frac{\partial k_1}{\partial M}+\frac{\partial k_2}{\partial M}}$.  
	
	Using this, and the fact that $k_1'(y) =\frac{\partial k_1}{\partial y} + \frac{\partial k_1}{\partial M}M'(y)$, the rest of the argument is a straightforward calculation.  
	
	For ease of notation, we set \[A_i = \int_{X^i_=(y, k_i)}\Big[\frac{D^2_{yy}c_i}{|D^2_{x_iy}c_i|}(x_i,y)\Big]\bar \mu_i(x_i)d\mathcal H^{m_i-1}(x_i)\] and \[B_i=\int_{X^i_=(y, k_i)}\Big[\frac{1}{|D^2_{x_iy}c_i|}(x_i,y)\Big]\bar \mu_i(x_i)d\mathcal H^{m_i-1}(x_i).\]  We then have
	\begin{eqnarray*}
		k_1'(y) &=& \frac{A_1}{B_1} - \frac{1}{B_1}\frac{\frac{A_1}{B_1}+\frac{A_2}{B_2}}{\frac{1}{B_1} +\frac{1}{B_2}}\\
		&=&\frac{A_1}{B_1} - \frac{1}{B_1}\frac{A_1B_2+A_2B_1}{B_1+B_2}\\
		&=&\frac{A_1(B_1+B_2) -A_1B_2-A_2B_1}{B_1(B_1+B_2)}\\
		&=&\frac{A_1-A_2}{B_1+B_2}.
	\end{eqnarray*}
	Therefore, a sufficient condition for nestedness is that, for all $\bar x_1 \in X_=^1(y, k_1)$
\begin{equation}
\label{cond}
\frac{A_1-A_2}{B_1+B_2} -D^2_{yy}c_1(\bar x_1,y) < 0, 
\end{equation}	
where
\begin{equation*}
  \begin{split}
  A_1-A_2 = &\int_{X^1_=(y, k_1)}\left[\frac{D^2_{yy}c_1}{|D^2_{x_1y}c_1|}(x_1,y)\right]\bar \mu_1(x_1)d\mathcal H^{m_1-1}(x_1)\\& -\int_{X^2_=(y, k_2)}\left[\frac{D^2_{yy}c_2}{|D^2_{x_2y}c_2|}(x_2,y)\right]\bar \mu_2(x_2)d\mathcal H^{m_2-1}(x_2)
  \end{split}
 \end{equation*}
 and
 \begin{equation*}
 \begin{split}
 B_1+B_2&=\int_{X^1_=(y, k_i)}\left[\frac{1}{|D^2_{x_1y}c_1|}(x_1,y)\right]\bar \mu_1(x_1)d\mathcal H^{m_1-1}(x_1)\\& +\int_{X^2_=(y, k_2)}\left[\frac{1}{|D^2_{x_2y}c_2|}(x_2,y)\right]\bar \mu_2(x_2)d\mathcal H^{m_2-1}(x_2).
 \end{split}
 \end{equation*}	
Thus, multiplying \eqref{cond} by $B_1+B_2$, we obtain \eqref{eqn: condition for nested hedonic equilibrium}.
\end{proof}
The lemma then yields the following result.
\begin{corollary}
	Assume that $c_1$ and $c_2$ are uniformly convex and that both the condition in the previous lemma, and the analogous condition obtained by reversing roles of $i=1$ and $i=2$, hold.  Then problem \eqref{eqn: hedonic problem}  is hedonically nested. 
\end{corollary}
\begin{remark}
	The first and third term in \eqref{eqn: condition for nested hedonic equilibrium} represent the difference between a weighted average of the positive function $D^2_{yy}c_1$ over the potential level set $X^1_=(y_1,k_1)$ and its value at a particular point, scaled by the total weighted mass $\int_{X^1_=(y, k_1)}\frac{1}{|D^2_{x_1y}c_1|}(x_1,y)\bar \mu_1(x_1)d\mathcal H^{m_1-1}(x_1)$ of that level set.  The second and fourth terms are both negative and have related interpretations.  One can ensure \eqref{eqn: condition for nested hedonic equilibrium} holds by imposing bounds on the variation of the $D^2_{yy}c_i$ over the potential level sets, and that the differences between the total masses 
	\begin{equation}\label{eqn: level set mass}
	\int_{X^i_=(y, k_i)}\left[\frac{1}{|D^2_{x_iy}c_i|}(x_i,y)\right]\bar \mu_i(x_i)d\mathcal H^{m_i-1}(x_i)
	\end{equation}
	are not too large whenever the super-level sets $X^i_\geq(y_i, k_1)$ have the same mass.
	
	As a very basic example,  if $\mu_1 =\mu_2$ and 
	$c_1=c_2:=c(I(x),y)$ are both the same index cost, the conditions hold automatically, since equality between the costs and marginals implies that the level set masses are identical and 
	the index form ensures that $c_{yy}$ does not vary throughout any level set $X_=(y,k)$.  Perturbations of the form $c_i(x_i,y) = c(I(x),y) +\epsilon \bar c_i(x_i,y)$ will still be nested for small $\epsilon$ and smooth $\bar c_i$; the variation of $D^2_{yy}c_i$ will be small, and so the first and fourth term
	 as will the difference in the masses \eqref{eqn: level set mass} of mass splitting level sets.
\end{remark}

\appendix
\section{A improved bound implying nestedness for the example in Corollary \ref{cor: disk to circle full}}
\label{App}

In this appendix we  refine the computations in Corollary \ref{cor: disk to circle full} and improve the bound on $\bar y$ by exploiting  local information on $k(y)$.\\
We, firstly, recall that in this example the level sets are given by
\[ x_1\sin(y)-x_2\cos(y)=k(y), \]
which can be also be re-written as
\begin{equation}
\label{line}
x_2=\tan(y)x_1-\dfrac{k(y)}{\cos(y)}. 
\end{equation}
Note, now, that for $y=0$, the mass splitting level set corresponds to $k(0) =0$, and so $X_=(0,k(0)) =X_=(0,0)$ is the $x_1$ axis.  
This implies that we cannot have nestedness (at least not for the full closed interval $[0,\bar y]$; that is, \eqref{eqn: 1-d nestedness} will fail with $y_0=0$) unless $k(y)<0$, in which case $X_=(y,k(y))$ intersects the boundary of the quarter disk on the $x_2$ axis.
If $k(y)>0$, it is clear by \eqref{line} that $X_=(y,k(y))$ would intersect the $X_=(0,k(0))$ and nestedeness would fail.
Therefore, we take $k(y) <0$, and   $X_=(y_0,k(y))$ intersects the $x_2$-axis at $(0,-\dfrac{k(y)}{\cos(y)})$.  Nestedness is equivalent to this point $-\dfrac{k(y)}{\cos(y)}$ of intersection being monotone increasing.  Now,  the differential equation  \eqref{MAmulti2one}\footnote{Note that this equation holds even when the model is not nested, by \eqref{eqn: def of k}, as the left hand side is simply the derivative of $\mu(X_{\geq}(y,k(y)))$, as in \cite{PassM2one} (in the non-nested case, however, $k(y)$ is not the derivative of a $c$-concave potential).} for $k(y)$ reads
$$
\int_{X_=(y,k(y))}[x\cdot(\cos(y),\sin(y)) -k'(y)]\frac{4}{\pi}d\mathcal{H}^1(x)=\bar \nu(y)
$$
We can parameterize  $X_\geq(y,k(y)=\{(0,-\frac{k(y)}{\cos(y)}) +t(\cos(y),\sin(y)): t\in(0,L(y,k(y)))\}$, where $L(y,k) = \sqrt{1-k^2} +k\tan(y)$ is the length of the level set $X_=(y,k)$; evaluating the integral above then yields



$$
[-k(y)\tan(y)-k'(y)]L(y,k(y))\frac{4}{\pi} +\frac{L^2(Y,k(y))}{2}\frac{4}{\pi} =\bar \nu(y).
$$
Identifying the first term on the left hand side above as $\cos(y)$ times the derivative of $\frac{-k(y)}{\cos(y)}$, we have that nestedness is equivalent to 
\begin{equation}\label{eqn: sharper bound for nestedenss}
\frac{2}{\pi}L^2(y,k(y)) \leq \bar \nu(y)
\end{equation}
for all $y$. Note that $L(y,(k(y))) \leq 1$, with equality if $k(y)=0$.  This inequality gives the sufficient condition \eqref{eqn: sufficient lower bound} for nestedness in Example \ref{ex: disk to arc}.  However, $k(0) =0$, and our bound \eqref{eqn: lower bound on CN density} on $\bar \nu$ is strongest at this point.  As $y$ increases, $k(y)$ will increase  and so the required bound \eqref{eqn: sharper bound for nestedenss} needed for nestedness becomes less stringent.

Using the bound \eqref{eqn: lower bound on CN density}, 
 we get the following sufficient condition for nestedness:

$$
e^yL^2(y,k(y)) \leq \frac{\pi}{2(e^{\bar y} -1)}
$$. 


It is straightforward to find $$\mu(X_\geq(y,k)) =[-\frac{k}{2} L(y,k) +\frac{y-\arcsin(k(y))}{2}]\frac{4}{\pi},$$ since this set is the disjoint (up to negligible sets) union of a wedge and a triangle. Mass balance then implies that $$[-\frac{k(y)}{2} L(y,k(y)) +\frac{y-\arcsin(k)}{2}]\frac{4}{\pi}= \nu(0,y)=\int_0^{\bar y}\bar \nu(y)dy \geq  \frac{1-e^{-y}}{e^{\bar y}-1},$$ or
$$
-k(y) L(y,k(y)) +y-\arcsin(k(y))\geq  \frac{\pi}{2}\frac{1-e^{-y}}{e^{\bar y}-1} 
$$
Now, the function $k \mapsto -k L(y,k)+y-\arcsin(k)$ is clearly decreasing (since as $k$ gets smaller, $\mu(X_\geq(y,k))$ must increase); denote its inverse by $Z(y;\cdot)$, so that the above is equivalent to
$$
k(y) \leq Z(y;\frac{\pi}{2}\frac{1-e^{-y}}{e^{\bar y}-1} ).
$$
Since $L$ is increasing in $k$, we have that 
$$
L(y,k(y)) \leq L(y,Z(y;\frac{\pi}{2}\frac{1-e^{-y}}{e^{\bar y}-1} )) 
$$
and so a sufficient condition for nestedness is that, for all $y$:

$$
e^yL^2(y,Z(y;\frac{\pi}{2}\frac{1-e^{-y}}{e^{\bar y}-1} )) \leq  \frac{\pi}{2(e^{\bar y} -1)}.
$$

Since the left hand side is increasing in $\bar y$, its maximum over $y$ is also increasing in $\bar y$.  The left hand side is decreasing in $\bar y$.  Therefore, the inequality holds if and only if $\bar y \leq \tilde y$ for some $\tilde y$.  Numerically, we can determine that $\tilde y \approx 0.65806$.

\bibliographystyle{plain}

\bibliography{bibli}

\end{document}